\documentclass{amsart}
\usepackage{graphicx} 
\usepackage{amsthm, amsmath, amssymb, amsfonts}
\usepackage{tikz, tikz-cd}
\usepackage{todonotes}
\usepackage{color}
\usepackage{bm}
\usepackage{enumitem}

\definecolor{amaranth}{rgb}{0.9, 0.17, 0.31}

\usepackage{hyperref}
\hypersetup{
    colorlinks=true,
    citecolor=amaranth,
    linkcolor=black,
    filecolor=magenta,      
    urlcolor=blue,
}

\tikzstyle{wbullet}=[circle, draw=black, fill=white, thick, inner sep=0pt, minimum size=1.5mm]
\tikzstyle{bbullet}=[circle, draw=black, fill=black, inner sep=0pt, minimum size=1.5mm]
\tikzstyle{cross}=[circle, draw=black, fill=white, inner sep=0pt, minimum size=1.5mm]

\numberwithin{equation}{section}

\newtheorem{thm}{Theorem}[section]
\newtheorem{claim}[thm]{Claim}

\newtheorem{cor}[thm]{Corollary}
\newtheorem{lem}[thm]{Lemma}
\newtheorem{prop}[thm]{Proposition}

\theoremstyle{definition}

\newtheorem{defn}[thm]{Definition}

\newtheorem{nota}[thm]{Notation}

\newtheorem{rmk}[thm]{Remark}

\newcommand{\Def}{\mathrm{Def}}

\newcommand{\I}{\mathrm{I}}
\newcommand{\II}{\mathrm{II}}
\newcommand{\III}{\mathrm{III}}

\newcommand{\ns}{\mathrm{ns}}

\newcommand{\Sing}{\mathrm{Sing}}

\newcommand{\sss}{\mathrm{ss}}
\newcommand{\Supp}{\mathrm{Supp\,}}

\newcommand{\A}{\mathbb{A}}
\newcommand{\CC}{\mathbb{C}}
\newcommand{\DD}{\mathbb{D}}

\newcommand{\PP}{\mathbb{P}}
\newcommand{\QQ}{\mathbb{Q}}
\newcommand{\RR}{\mathbb{R}}
\newcommand{\ZZ}{\mathbb{Z}}

\newcommand{\sA}{\mathcal{A}}
\newcommand{\sB}{\mathcal{B}}
\newcommand{\sC}{\mathcal{C}}
\newcommand{\sD}{\mathcal{D}}

\newcommand{\sF}{\mathcal{F}}

\newcommand{\sK}{\mathcal{K}}

\newcommand{\sO}{\mathcal{O}}

\newcommand{\sR}{\mathcal{R}}
\newcommand{\sS}{\mathcal{S}}

\newcommand{\sX}{\mathcal{X}}

\newcommand{\cB}{\mathcal{B}}
\newcommand{\cC}{\mathcal{C}}
\newcommand{\cD}{\mathcal{D}}

\newcommand{\cK}{\mathcal{K}}

\newcommand{\cO}{\mathcal{O}}

\newcommand{\cR}{\mathcal{R}}

\newcommand{\cX}{\mathcal{X}}

\newcommand{\tD}{\widetilde{D}}
\newcommand{\tE}{\widetilde{E}}
\newcommand{\tF}{\widetilde{F}}

\newcommand{\tV}{\widetilde{V}}

\newcommand{\tX}{\widetilde{X}}

\newcommand{\oA}{\overline{A}}
\newcommand{\oB}{\overline{B}}

\newcommand{\oD}{\overline{D}}

\newcommand{\oF}{\overline{F}}

\newcommand{\oK}{\overline{K}}

\newcommand{\oM}{\overline{M}}

\newcommand{\oP}{\overline{P}}
\newcommand{\oQ}{\overline{Q}}
\newcommand{\oR}{\overline{R}}
\newcommand{\oS}{\overline{S}}

\newcommand{\oX}{\overline{X}}

\newcommand{\oZ}{\overline{Z}}

\newcommand\osS{\overline{\sS}}
\newcommand\osR{\overline{\sR}}

\newcommand\Lattice{\Lambda}

\newcommand{\oTheta}{\overline{\Theta}}

\newcommand{\codim}{\mathrm{codim}}
\newcommand{\Exc}{\mathrm{Exc}}

\newcommand{\Kth}{\mathrm{K3}}

\newcommand{\Pic}{\mathrm{Pic}}
\newcommand{\rBB}{\mathrm{BB}}

\newcommand{\rc}{\mathrm{rc}}
\newcommand{\red}{\mathrm{red}}

\newcommand\Bl{\operatorname{Bl}}

\newcommand\bA{\mathbb A}
\newcommand\bC{\mathbb C}
\newcommand\bP{\mathbb P}
\newcommand\bQ{\mathbb Q}
\newcommand\bR{\mathbb R}
\newcommand\bZ{\mathbb Z}

\newcommand{\hB}{\widehat{B}}

\newcommand{\hX}{\widehat{X}}

\newcommand{\Nef}{\operatorname{Nef}}
\newcommand\ba{{\bm{a}}}
\newcommand\SP{\operatorname{SP}}
\newcommand\Spec{\operatorname{Spec}}
\newcommand\Proj{\operatorname{Proj}}
\newcommand\rank{\operatorname{rank}}

\title[Moduli of stable surfaces with $p_g=1$ of minimal volume]%
{On the moduli space of stable surfaces with $p_g=1$ realizing the minimal volume}
\author{Valery Alexeev}
\address{Department of Mathematics, University of Georgia, Athens GA 30602, USA}
\email{valery@math.uga.edu}

\author{Wenfei Liu}
\address{School of Mathematical Sciences, Xiamen University, Siming South Road 422, Xiamen, Fujian Province, P.~R.~China}
\email{wliu@xmu.edu.cn}

\author{Matthias Sch\"utt}
\address{Institut f\"ur Algebraische Geometrie, Leibniz Universit\"at
  Hannover, \linebreak
  Welfengarten 1, 30167 Hannover, Germany\\ and
   Riemann Center for Geometry and Physics, Leibniz Universit\"at
  Hannover, Appelstrasse 2, 30167 Hannover, Germany}
\email{schuett@math.uni-hannover.de}

\date{March 29, 2026}

\begin{document}

\begin{abstract}
Let $M_1$ be the moduli space of the KSBA stable surfaces $X$ of geometric genus $p_g(X)=1$ realizing the minimal possible volume $K_X^2=\frac1{143}$. We show that $M_{1}$ is a $10$-dimensional weighted projective space. By a result of Brieskorn, it is isomorphic to the Baily--Borel compactification $\oF_\Lattice^{\rBB}$ of the moduli space of $\Lattice$-polarized K3 surfaces, where $\Lattice=\II_{1,9}\simeq U\oplus E_8$ is a unimodular lattice of signature $(1,9)$. 

More generally, let $M_c$ be the moduli space $M_c$ of KSBA stable surface
pairs $(X,B)$ with coefficients of $B$ belonging to a set $\sC\subset (0,1]$ such that $\sC\cup\{1\}$ attains a minimum, say $c$, and with $p_g(X)=1$, realizing the minimal possible volume $(K_X+B)^2=v(c)$. We show that  the reduced part $M_{c,\red}$  is independent of~$c$ and that for $c\le\frac7{13}$ or $c\geq \frac{6}{11}$, $M_c$ is isomorphic to $\oF_\Lattice^\rBB$. Up to normalization, we provide a geometric interpretation of this isomorphism using the Kulikov model.

We also verify the Viehweg hyperbolicity of the base of a Whitney equisingular family of stable surface pairs in $M_c$. 
\end{abstract}

\maketitle

{
\hypersetup{linkcolor=black}
\setcounter{tocdepth}{1}
\tableofcontents
}
\section{Introduction}
The (projective) moduli spaces of stable curves with given genus are probably the most studied varieties in algebraic geometry. Their success has inspired the development of the KSBA (Kollár–Shepherd-Barron–Alexeev) theory, which constructs compact moduli spaces for higher-dimensional stable varieties and pairs. These spaces parametrize projective varieties and pairs with semi-log-canonical (slc) singularities and an ample canonical class, providing a powerful framework for studying degenerations and compactifying moduli problems \cite{kollar1988threefolds-and-deformations, alexeev2006higher-dimensional, kollar2023families-of-varieties}. Despite their existence, explicit descriptions of these moduli spaces remain scarce for surfaces and higher-dimensional varieties. Known examples include compactifications for stable varieties related to products of curves \cite{vO05, vO06a, vO06b, Rol10, Liu12, BHP+13}, and certain surfaces of general type with $p_g=0$ like Campedelli and Burniat surfaces \cite{AP09, AH25}. On the other hand, in the pair setting, the KSBA moduli theory has found important applications in the compactification of the moduli spaces of curves \cite{Has03, Hac04}, abelian varieties \cite{Ale02}, del Pezzo surfaces \cite{HKT09, AB22},  K3 surfaces \cite{Laz16, Sch18, MS21, AT21, ABE22, AE23, AET23, AB23, AEH24}, Enriques surfaces \cite{Sch22, AEG+25}, and elliptic surfaces \cite{AB21, bejleri2024moduli-of-elliptic, ISZ25, IZ25}.


In general, it is often very difficult to classify completely stable varieties of a given volume in dimensions higher than one, making the corresponding moduli space — despite its existence — largely inaccessible. For example, the explicit compactification of the moduli spaces of quintic surfaces and $(1,2)$-surfaces is still partial (\cite{Ran17, DeV22, CFP+23, GPS+23}), despite the effort of several research groups. Even when the geometric structure of the stable varieties at hand is known, it is a delicate problem to determine the scheme (or stack) structure of the moduli space. Also, finer global properties of the moduli spaces, such as rationality and hyperbolicity, are usually out of reach. 

For the classification of a given class of stable varieties, one method is to look at those
whose volumes reach a certain minimum. For example, Horikawa classified surfaces of general type achieving the equality of the Noether inequality $K_S^2\geq 2p_g(S)-4$ (\cite{horikawa1976algebraic-surfaces}), and this class of surfaces is now called Horikawa surfaces. 
This classification has now been largely extended to dimension three by \cite{chen2020noether-inequality, hu2025algebraic-threefolds, chen2025noether-inequality}.
Allowing klt singularities, Tsuonoda--Zhang \cite{tsunoda1992noether-inequality} found the minimal volume for surfaces of log general type with given geometric genus $p_g\geq 2$; see \cite{liu2017minimal-volume} for an updated treatment and statement for the surfaces realizing the minimal volume. These surfaces naturally form an open subset of the corresponding moduli space of stable surfaces. Recently, \cite{bejleri2024moduli-of-elliptic} studied the stable degenerations of these surfaces and thus obtained more information on the irreducible components of the moduli space of stable surfaces.

In this paper, based on the classification result of \cite{liu2017minimal-volume}, we study the scheme structure of the moduli space of stable surfaces with $p_g=1$ realizing the minimal volume. In fact, we can be more general, dealing with stable surface pairs $(X, B)$. More precisely, let $\sC\subset(0,1]$ be a subset (possibly empty) such that $\sC\cup\{1\}$ attains a minimum, say $c$, and let $\sS(\sC, 1)$ be the set of isomorphism classes of stable surface pairs $(X, B)$ such that $p_g(X, B):=h^0(X, K_X+\lfloor B\rfloor)=1$ and the coefficients of the boundary divisor $B$ come from $\sC$. 
Then the set $\{(K_X+B)^2\mid (X, B)\in \sS(\sC,1)\}$ attains the minimum $v(c)$, which depends only on $c$ (see Theorem~\ref{thm: liu2017minimal-volume} for the precise value of $v(c)$). 
We are able to completely determine the reduced scheme structure of the moduli space  $M_c$ of stable surface pairs $(X, B)\in \sS(\sC,1)$ such that $(K_X+B)^2=v(c)$; in case $c$ is close to $0$ or $1$, this is already the full scheme structure.
\begin{thm}\label{thm: main}
The reduced part $M_{c,\red}$ of the moduli space $M_c$ is isomorphic to the $10$-dimensional weighted projective space
\[
\bP:=\bP(4, 10, 12, 16, 18, 22, 24, 28, 30, 36, 42).
\]
If $c\leq \frac{7}{13}$ or $c\geq \frac{6}{11}$, then $M_c\cong \PP$ as stacks. In particular, $M_{c,\red}$ is independent of $c$, and it is rational and has cyclic quotient singularities.
\end{thm}
Combining with Brieskorn \cite[Theorem 5]{brieskorn1981unfolding-of-exceptional}, we have:
\begin{cor}\label{cor: main}
If  $c\leq \frac{7}{13}$ or $c\geq \frac{6}{11}$, then the moduli space $M_c$ is isomorphic to $\oF_{\Lattice}^\rBB$, the Baily--Borel compactification of the moduli space $F_\Lambda$ of $\Lattice$-polarized K3 surfaces for $\Lattice=\II_{1,9}\simeq U\oplus E_8$ (equivalently, $\Lattice^\perp=\II_{2,10}$). 
\end{cor}
In Section~\ref{sec:proof-by-degenerations}, we construct nef Kulikov models for the families of K3 surfaces in $F_\Lattice$, and show that the identification $M_c\cong \oF^\rBB_\Lattice$ for $0<c\leq \frac{7}{13}$ is an instance of a KSBA compactification of moduli spaces of K3 surface pairs $(X,\epsilon  R)$ with a recognizable divisor $R$ in the sense of \cite{AE23}. 

\begin{rmk}
The space $\PP$ can also be interpreted as the fine moduli space of $A$-polarized surfaces in the sense of Looijenga (\cite[Proposition~A.6]{looijenga1984smoothing-components}), where $A=\CC[x,y,z]/(z^2+y^3+x^7)$ is a graded $\CC$-algebra with $\deg x=6$, $\deg y=14$, and $\deg z=21$.  These coincidences could well serve as an instance of Leibniz' "preestablished harmony" (see the last sentence of \cite{brieskorn1981unfolding-of-exceptional}).
\end{rmk}




For the proof of Theorem~\ref{thm: main}, we establish,  for any $0\leq c, c'\leq 1$,  a bijection between the sets of families of stable surface pairs in $M_{c}$ and $M_{c'}$ over a reduced base scheme. It follows that $M_{c,\red}$ is independent of $c$ (Theorem~\ref{thm: same reduced}). 

In the cases $0<c\leq \frac{7}{13}$ and $\frac{6}{11}\leq c\leq 1$, the determination of the scheme structure of $M_c$ is via the study of the log canonical rings of the stable surface pairs.
\begin{thm}\label{thm: main2}
Let $(X, B)$ be a stable surface pair in $M_c$.
  \begin{enumerate}
      \item If $\frac{6}{11}\leq c\leq 1$, then $B=0$ and the canonical ring of $X$ is
    \[
    R(X, K_X) = \bigoplus_{m\geq 0} H^0(X, m K_X)\cong \CC[w,x,y,z]/(f)
    \]
    where $(\deg w,\,\deg x,\, \deg y,\, \deg z)=(1, 11, 26, 39)$, and $f(w,x,y,z)=z^2+y^3+x^7w+g(w,x,y)$ is a weighted homogeneous polynomial of degree $78$, with $g$ nonzero, not containing any monomial divisible by $y^2$ or $x^6$.
    \item If $0<c\le\frac{7}{13}$, then $(X, \lceil B\rceil)$ is a stable surface pair 
    with log canonical ring 
    \[
    R(X, K_X+\lceil B\rceil) = 
    \bigoplus_{m\geq 0} H^0(X, m (K_X+\lceil B\rceil))\cong \CC[w,x,y,z]/(f)
    \]
    where $(\deg w, \deg x,\, \deg y,\, \deg z)=(1, 6, 14, 21)$, and $f(w,x,y,z)= z^2+y^3+ x^7 + g(w,x,y)$ is a weighted homogeneous polynomial of degree $42$, with $g$ nonzero, not containing any monomial divisible by $y^2$ or $x^6$, and $\lceil B\rceil=(w=f=0)$.
  \end{enumerate}
\end{thm}
Theorem~\ref{thm: main2} follows from the combination of Propositions~\ref{prop: canonical ring c small}, \ref{prop: singularity analysis c small}, \ref{prop: canonical ring c large}, and \ref{prop: singularity analysis c large}, whose proofs use Blache's Riemann--Roch formula for singular surfaces and a singularity analysis for elliptic surfaces in Weierstrass normal form. One can see from the equations of Theorem~\ref{thm: main2} that our stable surface pairs can be viewed as deformations of appropriate compactifications of the affine surface $(z^2+y^3+x^7=0)\subset \CC^3$. In case $c\leq \frac{7}{13}$, by raising the coefficient of boundary divisor of $(X, B)\in M_c$ to $1$, we recover Looijenga's $A$-polarized surfaces with $A=\CC[x,y,z]/(z^2+y^3+x^7)$ with $(\deg x, \deg y, \deg z)=(6,14,21) $. 

The determination of the moduli space $M_c$ as a scheme (or a stack) requires a version of Theorem~\ref{thm: main2} over any base scheme. For this, as was already observed by Pinkham \cite{Pin77}, Brieskorn \cite{brieskorn1981unfolding-of-exceptional}, and Looijenga \cite{looijenga1984smoothing-components}, we need the fact that the negative-weight part $\Def_-(S_{237})\cong \CC^{11}$ of the semiuniversal deformation space of the quasihomogeneous $S_{237}$-singularity $(z^2+y^3+x^7=0)\subset\CC^3$ is in a natural bijection to the deformation space of the equation $f$ with normalized principal term $z^2+y^3+x^7$ (resp.~$z^2+y^3+x^7w$) in the case $0<c\leq \frac{7}{13}$ (resp.~$\frac{6}{11}\leq c\leq 1$). The upshot is that any nontrivial negative-weight deformation of the $S_{237}$-singularity, corresponding to a nonzero vector in $\Def_-(S_{237})\cong \CC^{11}$, has at most log canonical singularities. Our stable surface pairs compactify the affine surfaces by adding a hyperplane at infinity, and we obtain without further modification a family of stable surface pairs over the whole $\Def_-(S_{237})\setminus\{0\}$. Using the normalized equations as in Theorem~\ref{thm: main2}, passing to the moduli stack $M_c$ turns out to be a matter of quotienting out $\CC^*$ with appropriate weights:
\[
M_c\cong (\Def_-(S_{237})-\{0\})/\CC^* \cong \bP(4, 10, 12, 16, 18, 22, 24, 28, 30, 36, 42).
\]
see Sections~\ref{sec: M_c with c small} and \ref{sec: M_c with c large} for details.

\medskip
Finally, in Section~\ref{sec: hyperbolicity}, we verify the Viehweg hyperbolicity for (Whitney) equisingular families of stable surfaces in $M_c$, as proposed in \cite{park2022viehweg-hyperbolicity}.
\begin{thm}\label{thm: main3}
Let $V$ be a smooth quasi-projective variety, and let $f\colon \sX\rightarrow V$ be a family of equisingular stable surfaces from $M_c$. If $f$ is of maximal variation, that is, the moduli map $V\rightarrow M_c$ is generically finite onto the image, then $V$ is of log general type.
\end{thm}
Indeed, given a (Whitney) equisingular family $f\colon \sX\rightarrow V$ in $M_c$ over a smooth quasi-projective variety $V$, we may resolve the singularities simultaneously and obtain a smooth family of K3 surfaces or elliptic curves. If $f$ is of maximal variation, then so is the corresponding family of K3 surfaces or elliptic curves. It follows then that $V$ is of log general type by \cite{popa2017viehweg-hyperbolicity}; see Theorem~\ref{thm: hyperbolicity}.

\medskip

\noindent{\bf Notation and Conventions.} We work over the complex numbers $\CC$.

A variety means an integral scheme of finite type (over $\CC$). 

A curve (resp.~surface) means a connected reduced scheme of finite type, all of whose irreducible components have dimension one (resp.~two).

Let $X$ be a scheme (over $\CC$), we denote by $X_{\red}$ its reduced part and by $X_{\Sing}$ the singular locus of $X$.
For a cycle with $\RR$-coefficients $B=\sum_i b_i B_i$ on a reduced scheme $X$, its round-up is defined to be $\lceil B\rceil:=\sum \lceil b_i\rceil B_i$, where $\lceil b_i\rceil$ is the smallest integer satisfying $\lceil b_i\rceil \geq b_i$, and the round-down $\lfloor B\rfloor:=\sum \lfloor b_i\rfloor B_i$ is defined similarly. The reduced part of $B$ is defined as $B_\red:=\sum_{b_i\neq 0} B_i$. The cycle $B=\sum_i b_iB_i$ is called a $\RR$-divisor on $X$ if $B_i$ is not contained in the singular locus $\Sing(X)$ of $X$ and if the codimension $\codim_X(B_i)=1$ for each $i$. An $\RR$-divisor is \emph{effective} if each of the coefficients $b_i$ is nonnegative.

For a birational class $\sS$ of projective normal surfaces, we choose a canonical divisor $K_X$ for each $X\in \sS$ so that if there is a birational morphism $f\colon X_1\rightarrow X_2$ 
in $\sS$
then $K_{X_2}=f_*K_{X_1}$.

We use Kodaira's notation, $\I_k, \II, \II^*$ and so on for the type of fibers of an elliptic fibration.

\medskip

\noindent{\bf Acknowledgment.} The project was initiated while the second author was visiting Leibniz Universit\"at Hannover in July 2024; he enjoyed the academic atmosphere as well as the hospitality at the Institut f\"ur algebraische Geometrie. The first two authors would like to thank the organizers of the 2025 SCMS Algebraic Geometry Summer School at Fudan University, where a major part of the work has been done. The first author was partially supported by the NSF under DMS-2501855. The second author was partially supported by the NSFC (No.~12571046).
The third author's research is partly conducted in the framework of the research training
group GRK 2965: From Geometry to Numbers,
funded by DFG.

\section{Preliminaries}
\subsection{Stable surface pairs and stable families}
Recall that a demi-normal scheme is a (reduced) scheme that is $S_2$ and whose codimension 1 points are either regular points or nodes (\cite[Definition~5.1]{kollar2013singularities-of-the-minimal}).
 A \emph{surface pair} $(X,B)$ consists of a connected demi-normal surface $X$ and an effective Weil $\RR$-divisor $B$ such that $\Supp B$ does not contain any irreducible component of the non-normal locus of $X$ and that $K_X+B$ is $\RR$-Cartier. When $B=0$, we usually omit the word "pair". The surface pair $(X, B)$ is called normal (resp.~projective) if $X$ is so; it is called \emph{rational} (resp.~\emph{K3}) if $X$ is birational to $\PP^2$ (resp.~to a K3 surface). 

For a normal surface pair $(X,B)$, we may take a log resolution $f\colon Y\rightarrow X$ so that $\Exc(f)\cup f^{-1}_*B$ has simple normal crossing support, where $\Exc(f)$ is the exceptional locus of $f$ and $f^{-1}_*B$ is the strict transform of $B$ on $Y$. We
can write
\[
K_Y + B_Y = f^*(K_X+B)
\]
such that $f_*B_Y=B$. One says that $(X, B)$ is \emph{log canonical} (abbreviated as lc) if the coefficients of $B_Y$ are at most $1$.

For a general (possibly non-normal) surface pair $(X, B)$, we may take the normalization $\pi\colon \oX\rightarrow X$. Let $\oX=\bigsqcup_{i=1}^r\oX_i$ be the decomposition into irreducible (=connected) components. Let $\oD\subset\oX$ be the conductor divisor and $\oD_i=\oD\cap \oX_i$. Let $\oB$ be the $\RR$-divisor on $\oX$ defined by
\[
\pi^*(K_X+B)=K_{\oX}+\oD+\oB.
\]
Then each $(\oX_i, \oD_i+\oB_i)$ is a normal surface pair, where $\oB_i=\oB|_{\oX_i}$. We say that $(X, B)$ is \emph{semi log canonical} (abbreviated as slc) if $(\oX_i, \oD_i+\oB_i)$ is log canonical for each $1\leq i\leq r$.
\begin{defn}
 A surface pair $(X,B)$ is called \emph{stable} if it is projective and has semi log canonical singularities, and if $K_X+B$ is an ample $\RR$-Cartier divisor.
\end{defn}

For any coefficient vector $\ba=(a_1,\dotsc,a_r)\in (0,1]^r$ and a positive real number $v$, called the volume, 
there is a projective coarse moduli space $\SP(\ba,2,v)$ of stable surface pairs $(X,B)$ satisfying the following conditions (\cite[Theorem~8.1]{kollar2023families-of-varieties}):
\begin{enumerate}
    \item $(K_X+B)^2=v$;
    \item $B=\sum_{i=1}^r a_iD_i$ for effective nonzero $\bZ$-divisors $D_i$.
\end{enumerate}
Noting that $p_g(\sX_s)=h^2(\sX_s, \sO_{\sX_s})$ is locally constant in $s$ in a stable family $f\colon\cX\to S$ by \cite[Theorem~2.62]{kollar2023families-of-varieties}, there is an open and closed subscheme $\SP(\ba,2,v,p)\subset \SP(\ba,2,v)$ in which every surface satisfies the condition $p_g(X):=h^0(X, K_X)=p$.

We will use a special case of this moduli space, when the geometric genus is $p_g(X)=1$ and the volume $v=v(c)$ is minimal possible for a given set $\cC$ of coefficients
with $\min(\sC\cup\{1\})=c$.
In Theorem~\ref{thm: liu2017minimal-volume} one can find the explicit value of $v(c)$, depending on the following  three cases:
\begin{enumerate}
    \item[(1)] $0<c\le \frac7{13}$, in which case $B=cD$ for an irreducible and reduced divisor~$D$, so that $D = B_\red$, and in addition one has $K_X=0$.
    \item[(2)] $\frac7{13} <c< \frac6{11}$, in which case $B=cD$ for an irreducible and reduced divisor~$D$, so that $D = B_\red$. 
    \item[(3)] $\frac6{11}\le c\le 1$, in which case $B=0$.
\end{enumerate}

For simplicity of notation, we denote $M_c:=\SP((c),2,v(c),1)$ in the first two cases, resp.~$M_c:=\SP(\emptyset,2,v(c),1)$ in the third case. Let us emphasize that there are essentially only three spaces $M_c$ here: for $c_1,c_2$ in the same interval one has a natural isomorphism $M_{c_1}=M_{c_2}$.

In general, defining a family $f\colon (\cX,\cB)\to S$ of stable pairs with a nonzero boundary $B$ is delicate, with the main difficulty being that generally the families of divisors $\cD_i\to S$ need not be flat, and if $\cD_i$ are considered to be closed subschemes of $\cX$ then some fibers may acquire embedded components. None of these complications appear in our limited situation, so we can define quite simply:

\begin{defn}
    A family $f\colon (\cX,\cB=c\cD)\to S$, resp. $f\colon (\cX,\cB=0)\to S$
    in $M_c$,
    is a flat morphism $\cX\to S$ such that 
    \begin{enumerate}
        \item $K_{\cX/S}$ is a relative $\bQ$-Cartier divisor, 
        \item $\cD$ is a relative $\bQ$-Cartier divisor, and
        \item every fiber $(X,B)$ is a stable pair
        in $M_c$
        of volume $v(c)$  with $p_g(X)=1$.
    \end{enumerate}
\end{defn}
In the Calabi-Yau case (1) $0<c\le \frac7{13}$ above, this holds by \cite[Theorem~3.13]{AET23}
(see also \cite{kollar2020moduli-of-polarized, birkar2023geometry-polarized}
for a higher-dimensional version). In either case (2) or (3), we can use the A-stability of \cite[Section~6.4]{kollar2023families-of-varieties} or AFI-stability of \cite[Section~8.3]{kollar2023families-of-varieties}. In the intermediate case $\frac{7}{13}<c<\frac{6}{11}$, note that $c>\frac12$, so the family of divisors $\cD\to S$ is flat by \cite[Section~2.7]{kollar2023families-of-varieties}.

\subsection{Riemann--Roch formula for projective normal surfaces}
In this subsection we recall Blache's Riemann--Roch theorem for complex projective normal surfaces.
\begin{thm}[{\cite[Theorem~1.2]{blache1995riemann-roch}}]\label{thm: Blache}
    Let $X$ be a projective normal surface over the complex numbers $\CC$, and $D$ a Weil divisor on $X$. Let $\pi\colon \tX\rightarrow X$ be a resolution of singularities. Then
    \begin{equation}\label{eq: Blache}
    \chi(X, \sO_X(D)) = \chi(X, \sO_X)+\frac{1}{2}(D-K_X)\cdot D + \sum_{x\in X_\Sing} \delta_x(D)
    \end{equation}
    where for $x\in X_\Sing$, $\delta_x(D)$ is the correction term (to the usual Riemann--Roch formula) depending only on the analytically local type of $x\in X$, and it is computed by
    \begin{equation}\label{eq: delta D}
    \delta_x(D) = -\frac{1}{2}\{\pi^*D\}_x\left(\lfloor \pi^*D \rfloor_x - K_{\tX}\right)        
    \end{equation}
    where $\{\pi^*D\}_x$ (resp.~$\lfloor \pi^*D \rfloor_x$) denotes the fractional (resp.~integral) part of the $\QQ$-divisor $\pi^*D$ over $x$. 
\end{thm}
We list some easy properties of the correction term defined as in \eqref{eq: delta D}.
\begin{lem}\label{lem: delta D}
Let the notation be as in Theorem~\ref{thm: Blache}. Then the following holds for the function $\delta_x(\cdot)$:
\begin{enumerate}
    \item (periodicity) Suppose that $D$ is Cartier at $x$. If  $E$ is another Weil divisor, then $\delta_x(D+E)=\delta_x(E)$. In particular, $\delta_x(D)=0$.
    \item (Serre duality) $\delta_x(D) = \delta_x(K_X-D)$.
    \item (\cite[Proposition~5.2]{blache1995riemann-roch}) Writing $\pi^*K_X=K_{\tX}+B_{\tX}$, where $B_{\tX}$ is a uniquely determined $\pi$-exceptional $\QQ$-divisor, we have 
    \[
    \delta_x(nK_X) =\frac{1}{2}\{nB_{\tX}\}_x\cdot\left(\{nB_{\tX}\}_x + K_{\tX}\right)
    \]
\end{enumerate} 
\end{lem}

\begin{rmk}\label{rmk: delta plurigenera}
  If a Weil divisor $D$ has Cartier index $r$ at $x$, then one needs only to compute $\delta_x(nD)$ for $1\leq n\leq r-1$, and the values for the other $n$ follow directly by periodicity as in Lemma~\ref{lem: delta D} (i). If furthermore $K_X$ is Cartier at $x$, then one needs only to compute $\delta_x(nD)$ for $1\leq n\leq \lfloor \frac{r}{2}\rfloor$, since for a general positive integer $n=mr+d$ with $0\leq d\leq r-1$, one has by Lemma~\ref{lem: delta D}, 
    \[
    \delta_x(nD) = \delta_x(dD) = \delta_x(K_X-dD) =\delta_x((r-d)D) .
    \]   

Similarly, for pluricanonical divisors $nK_X$ with $n=mr+d$, where $r$ is the Cartier index of $K_X$ at $x$ and $0\leq d\leq r-1$, we have
 \begin{equation}\label{eq: delta nK}
    \delta_x(n K_X) = \delta_x(dK_X) = \delta_x(K_X-dK_X) =\delta_x((r+1-d)K_X) .     
 \end{equation}
Thus it suffices to compute $\delta_x(dK_X)$ for $0\leq d \leq \lfloor \frac{r+1}{2}\rfloor$, and the rest can be deduced via the equalities in \eqref{eq: delta nK}.
\end{rmk}

\section{Stable surfaces with \texorpdfstring{$p_g=1$}{}, achieving the minimal volume}
\label{sec: characterize X}

In this section, we recall in Theorem~\ref{thm: liu2017minimal-volume} the characterization of stable surface pairs $(X, B)$ with $p_g(X,B)=1$ and minimal possible volume, given in \cite{liu2017minimal-volume}. Then we refine it by showing that there is a genus 1 fibration on a smooth model $Z$ of $X$ (Lemma~\ref{lem: elliptic fibration}). In case $X$ is rational, we can even classify the relatively minimal model $\oZ$ of  $Z$ over $\PP^1$ into two isomorphism classes (Lemma~\ref{lem: X22 X211}). 

Let us first introduce some notation.

\begin{nota}\label{nota: resolve X}
Let $\sC\subset (0,1]$ be a subset, possibly empty, such that $\sC\cup\{1\}$ attains the minimum, say $c=\min\left(\sC\cup\{1\}\right)$. Let $\sS(\sC,1)$ be the set of
isomorphism classes of normal stable surface pairs
$(X, B)$ such that
\begin{itemize}
    \item the coefficients of $B$ come from $\sC$,
    \item $p_g(X, B):=\dim H^0(X, K_X+\lfloor B\rfloor) = 1$.
\end{itemize}
 Consider the following diagram:
\begin{equation}\label{diag: triangle}
\begin{split}
\begin{tikzpicture}
\node (tX) at (0,0) {$({\tX}, B_{\tX})$};
\node (X)[below left=1cm of  tX] {$(X, B)$};
\node (Z)[below right=1cm of  tX] {$(Z,B_Z)$};
\draw[->](tX.south west)--(X) node[above, midway]{$\pi$};
\draw[->](tX.south east)--(Z)node[above, midway]{$\mu$};
\end{tikzpicture} 
\end{split}
\end{equation}
where $\pi\colon\tX\rightarrow X$ is the minimal resolution of singularities, $B_{\tX}$ is the effective $\RR$-divisor on $\tX$ such that $\pi^*(K_X+B)=K_{\tX}+B_{\tX}$ and $\pi_*B_{\tX} = B$, and $\mu\colon \tX\rightarrow Z$ is the minimal model of $(\tX, \lfloor B_{\tX}\rfloor)$.

Let $B_{\tX}=B_{\tX}^\sss + B_{\tX}^\ns$ be the decomposition into the semistable part  $B_{\tX}^\sss$ and the non-semistable part $B_{\tX}^\ns$ (\cite[Definition~4.4]{liu2017minimal-volume}),
and denote
\[
B_Z=\mu_*B_{\tX},\quad B_Z^{\sss} = \mu_*B_{\tX}^\sss,\quad B_Z^{\ns} = \mu_*B_{\tX}^\ns.
\]
\end{nota}

\begin{thm}[{\cite[Theorem~1.1 and Proposition~7.13]{liu2017minimal-volume}}]\label{thm: liu2017minimal-volume}
Keep Notation~\ref{nota: resolve X}. 
(1) The set $\{(K_X+B)^2 \mid (X,B)\in \sS(\sC,1)\}$ of volumes attains a minimum $v(c)$, which depends only on $c$. More precisely, we have
\begin{equation}\label{eq: min vol}
  v(c)=
\begin{cases}
\frac{1}{42}c^2, & \text{ if } 0<c\leq \frac{7}{13} \\
-\frac{11}{6} c^2 +2c -\frac{7}{13}, & \text{ if }  \frac{7}{13}<c<\frac{6}{11} \\
\frac{1}{143}, & \text{ if }  c\geq \frac{6}{11}.
\end{cases}  
\end{equation}

(2) Suppose that the volume of $(X, B)\in \sS(\sC,1)$ is $v(c)$. Then the coefficients of $B$ lie in $\{c\}$, and $p_g(X)=1$. Furthermore, the following holds for $(Z,B_Z)$:
\begin{itemize}
    \item $Z$ is either a smooth K3 surface or a smooth rational surface.
    \item $K_Z+B_Z^{\sss}\sim 0$, that is, $B_Z^{\sss}\in |-K_Z|$.
    \item $\Supp B_Z^{\sss} \cap \Supp B_Z^{\ns}=\emptyset$, and
   \item $\lceil B_Z^{\ns}\rceil$ is the union of $(-2)$-curves with following dual graph:
\begin{center}
  \begin{tikzpicture}[font=\small]
    \begin{scope}
        \foreach \x in {0, 1,...,7}  
    \draw (\x,0)--(\x+1,0);
    \draw (6,0)--(6,1);    
   \foreach \x in {0,...,8}
\node[wbullet, label=below: $\Theta_{\x}$] at (\x,0){};  
\node [wbullet, label=right: $\Theta_{9}$] at (6,1) {};  
   \end{scope} 
\end{tikzpicture}
  \end{center}
\end{itemize}
Moreover, the following holds for $(\tX, B_{\tX})$:
\begin{enumerate} 
  \item $(\tX, B_{\tX})$ and $(\tX, B_{\tX}^\sss + c\lceil B_{\tX}^\ns\rceil)$ have the same canonical model, which is $(X, B)$. 
  \item 
 If $Z$ is rational, then the map $\mu$ is an isomorphism over a neighborhood of $B_Z^\sss$, the semistable part $B_{\tX}^\sss$ of the boundary divisor on $\tX$ is connected, and it is contracted into a simple elliptic or cusp singularity on $X$.
      \item If $c\leq \frac{7}{13}$, then $\tX=Z$, and $X$ is obtained from $\tX$ by contracting 
      the components of $\lceil B_Z\rceil-\Theta_6$, which includes $B_{\tX}^\sss$, as well as the $(-2)$-curves not intersecting $\lceil B_Z\rceil$. The coefficients of the components in $B_{\tX}^\ns$ are specified in the dual graph as follows:
      \begin{center}
  \begin{tikzpicture}[font=\small]
    \begin{scope}
        \foreach \x in {0,1,...,7}  
    \draw (\x,0)--(\x+1,0);
    \draw (6,0)--(6,1);    
\node[wbullet, label=below: $\frac{2c}{7}\Theta_{1}$] at (1,0){};  
\node[wbullet, label=below: $\frac{3c}{7}\Theta_{2}$] at (2,0){};  
\node[wbullet, label=below: $\frac{4c}{7}\Theta_{3}$] at (3,0){};  
\node[wbullet, label=below: $\frac{5c}{7}\Theta_{4}$] at (4,0){};  
\node[wbullet, label=below: $\frac{6c}{7}\Theta_{5}$] at (5,0){};  
\node[wbullet, label=below: $c\Theta_{6}$] at (6,0){};  
\node[wbullet, label=below: $\frac{2c}{3}\Theta_{7}$] at (7,0){};  
\node[wbullet, label=below: $\frac{c}{3}\Theta_{8}$] at (8,0){};  
\node [wbullet, label=below: $\frac{c}{7}\Theta_0$] at (0,0) {}; 
\node [wbullet, label=right: $\frac{c}{2}\Theta_{9}$] at (6,1) {};  
   \end{scope} 
\end{tikzpicture}
  \end{center}
In this case, the boundary divisor $B$ on $X$ is
\[
B=\pi_*B_{\tX} = c\oTheta_6, \quad\text{where the curve } \oTheta_6=\pi_*\Theta_6
\]
is nonzero with irreducible support.

      \item If 
    $\frac{7}{13}<c<\frac{6}{11}$, then $\tX$ is obtained by blowing up $Z$ at the intersection point $\{R\}=\Theta_5\cap \Theta_6$. Let $E$ be the exceptional curve of $\mu\colon\tX\rightarrow Z$. Then $ B_{\tilde X}^\ns$ is the strict transform of $B_Z^\ns$, and the dual graph of $\lceil B_{\tilde X}^\ns\rceil+E$ is as follows:
\begin{center}
  \begin{tikzpicture}[font=\small,scale=1.1]
        \foreach \x in {0,1,...,8}  
    \draw (\x,0)--(\x+1,0);
    \draw (7,0)--(7,1);    
\node[wbullet, label=below: $\frac{1}{13} \widetilde\Theta_0$] at (0,0){}; 
\node[wbullet, label=below: $\frac{2}{13} \widetilde\Theta_1$] at (1,0){};  
\node[wbullet, label=below: $\frac{3}{13} \widetilde\Theta_2$] at (2,0){};  
\node[wbullet, label=below: $\frac{4}{13} \widetilde\Theta_3$] at (3,0){};  
\node[wbullet, label=below: $\frac{5}{13} \widetilde\Theta_4$] at (4,0){};  
\node[wbullet, label=below: $\frac{6}{13} \widetilde\Theta_5$] at (5,0){};  
\node[bbullet, label=below: $0E$] at (6,0){};  
\node[wbullet, label=below: $c\widetilde\Theta_6$] at (7,0){};  
\node[wbullet, label=below: $\frac{2c}{3} \widetilde\Theta_7$] at (8,0){};  
\node[wbullet, label=below: $\frac{c}{3} \widetilde\Theta_8$] at (9,0){};  
\node[wbullet, label=right: $\frac{c}{2} \widetilde\Theta_9$] at (7,1) {};
\end{tikzpicture}
  \end{center}
where $\widetilde\Theta_i \subset \tX$ are the strict transforms of $\Theta_i$, so we have
\[
\widetilde\Theta_i^2=-2 \quad \text{for $i\neq 5,6$}, \quad \widetilde\Theta_5^2 =\widetilde\Theta_6^2=-3,\quad E^2=-1
\]
and the fractional number beside
each node denotes the coefficient of the corresponding curve in $B_{\tX}^\ns$.
In this case, 
all the $\widetilde\Theta_i$ except for $i = 6$ are contracted by $\pi\colon\tX\rightarrow X$, and hence
the boundary $B$ on $X$ is 
\[
B=\pi_*B_{\tX} = c\pi_*\widetilde\Theta_6
\]
is nonzero with irreducible support.

\item If $\frac{6}{11}\leq c\leq 1$, then $\tX$ is the same as in (iv), and the dual graph of $\lceil B_{\tilde X}^\ns\rceil+E$ is obtained from that in (iv) by replacing the coefficients $c$ of $\widetilde{\Theta}_i$, $6\leq i\leq 9$, with $\frac{6}{11}$. Moreover, all the $\widetilde{\Theta}_i$, $0\leq i\leq 9$, are contracted by $\pi\colon \tX\rightarrow X$, so the boundary divisor on $X$ is zero.
  \end{enumerate}
\end{thm}

\begin{proof}
Only (ii) is not explicitly contained in \cite{liu2017minimal-volume}. The first statement of (ii) that $\mu$ is an isomorphism over a neighborhood of $B_Z^\sss=\mu_*B_{\tX}^\sss$ is explained in \cite[Notation 5.1]{liu2017minimal-volume}. Since $K_Z+B_Z^\sss\sim0$, we have a short exact sequence 
\begin{equation}\label{eq: ses}
0\rightarrow \sO_{Z}(K_Z)=\sO_{Z}(-B_Z^\sss)\rightarrow \sO_Z\rightarrow \sO_{B_Z^\sss}\rightarrow 0    
\end{equation}
Since $Z$ is rational, we have $H^1(Z, \sO_{Z}(K_Z))\cong H^1(Z, \sO_Z)=0$, and hence a surjection $H^0(Z, \sO_Z)\cong \CC\twoheadrightarrow H^0(B_Z^\sss, \sO_{B_Z^\sss})$ by the long exact sequence associated to \eqref{eq: ses}. It follows that $H^0(B_Z^\sss, \sO_{B_Z^\sss})\cong \CC$ and hence $B_Z^\sss$ is connected.

Since $\mu$ is an isomorphism over a neighborhood of $B_Z^\sss$ and $\Supp B_Z^{\sss} \cap \Supp B_Z^{\ns}=\emptyset$, we have 
\[
(K_{\tX}+B_{\tX})\cdot{B_{\tX}^\sss}=(K_Z+B_Z)\cdot B_Z^\sss=(K_Z+B_Z^\sss)\cdot B_Z^\sss = 0.
\]
It follows that $B_{\tX}^\sss$ is contracted to a point on the canonical model $X$. Since $B_{\tX}^\sss$ is semistable and hence has positive arithmetic genus, it is necessarily contracted to a simple elliptic or cusp singularity on $X$. 
\end{proof}

\begin{defn}\label{defn:types123}
The stable surface pair $(X, B)$ in Theorem~\ref{thm: liu2017minimal-volume} (2) is called 
\begin{itemize}
    \item of Type $\I$ if it is $K3$;
    \item of Type $\II$ if it is rational and has a simple elliptic singularity;
    \item of Type $\III$ if it is rational and has a cusp singularity.
\end{itemize}
We stratify the moduli space $M_c$ accordingly:
\begin{equation}\label{eq: stratification}
M_c = M_c^{\Kth} \sqcup M_c^{\rm rat},\text{ with } M_c^{\Kth}=M_c^\I
\text{ and } M_c^{\rm rat} = M_c^{\II} \sqcup M_c^{\III}.
\end{equation}
\end{defn}

\begin{defn}\label{defn:T237}
We will call the configuration of ten $(-2)$-curves $\Theta_0,\hdots,\Theta_9$ in part (iii) of Theorem~\ref{thm: liu2017minimal-volume} and its dual graph  \emph{the $T_{2,3,7}$-configuration.} Note that these $10$ curves generate the lattice $\Lattice\cong U \oplus  E_8$ which may alternatively be denoted by $E_{10}$.
\end{defn}

\begin{lem}\label{lem: non nonnormal}
    Let $(X, B)$ be a stable surface pair such that the coefficients of $B$ come from $\sC$, $p_g(X,B)=1$, and $(K_X+B)^2=v(c)$. Then $X$ is normal, and hence $(X,B)\in\sS(\sC,1)$.
\end{lem}
\begin{proof}
    Suppose on the contrary that $X$ is not normal, possibly with several irreducible components.
    Let $\nu\colon \oX\rightarrow X$ be the normalization. Let $\oD\subset \oX$ be the nonzero conductor divisor, and let $\oB$ be the $\RR$-divisor on $\oX$ defined by
    \[
    \nu^*(K_X+B)=K_{\oX}+\oD+\oB.
    \]
    Then there is a component $\oX^\prime$ of $\oX$ such that $p_g(\oX^\prime, \oB^\prime+\oD^\prime)>0$, where $\oB^\prime$ and $\oD^\prime$ are the restrictions of $\oB$ and $\oD$ to $\oX^\prime$ respectively. Since $\lfloor \oB^\prime+\oD^\prime\rfloor \geq \oD^\prime>0$, the classification of \cite{liu2017minimal-volume} gives
    \[
    (K_{\oX^\prime}+\oB^\prime+\oD^\prime)^2 >v(c).
    \]
    On the other hand,
    \[
    (K_{X}+B)^2 = (K_{\oX}+\oB+\oD)^2\geq (K_{\oX^\prime}+\oB^\prime+\oD^\prime)^2 >v(c),
    \]
where the inequality holds because the volume of $K_{\oX}+\oB+\oD$ is the sum of the volumes of its restrictions to the irreducible components of $\oX$. This is a contradiction to the assumption.
\end{proof}

\begin{thm}\label{thm: same reduced}
For $0<c\leq 1$, the isomorphism class of the reduced moduli space $M_{c,\red}$ does not depend on $c$.
\end{thm}

\begin{proof}
Take two real numbers $0<c, c'\leq 1$. First we establish a natural bijection between the closed points  of the moduli spaces $M_c$ and $M_{c'}$. 

 For $[(X, B)]\in M_c(\CC)$, let $A$ be the union of the non-Gorenstein singularities of $X$ and the singularities of $X$ lying on $B$. Then $A$ consists of (at most 3) rational singularities of $X$, so we may blow up them sufficiently many times, without having to take the normalization in the intermediate steps,\footnote{In fact, the blow-up of a rational surface singularity is still normal by \cite[Proposition~8.1]{lipman1969rational-singularities}.} to obtain a partial resolution $\hat\pi\colon \hX\rightarrow X$ such that the following holds:
\begin{enumerate}
    \item $\hat\pi$ is an isomorphism over the open set $X\setminus A$;
    \item $\hX$ is smooth over a neighborhood of $A$;
    \item the center on $\hX$ of the valuation corresponding to the black dot $E$ in the dual graph of Theorem~\ref{thm: liu2017minimal-volume} (iv) is a curve. 
\end{enumerate}
Write $\hat\pi^*(K_X+B) = K_{\hX} + \hB$ with $\hat\pi_*\hB= B$, and let $\hB^{>0}$ be the part of $\hB$ with positive coefficients. Then, by Theorem~\ref{thm: liu2017minimal-volume}, the canonical model $(X', B')$ of $(\hX, c'\lceil \hB^{>0} \rceil)$ lands in $M_{c'}$. Thus, we obtain a map of sets $\sigma_{c,c'}(\CC)\colon M_c(\CC)\rightarrow M_{c'}(\CC)$. 

We can apply the same construction with the roles of $c$ and $c'$ reversed, and obtain a map $\sigma_{c',c}(\CC)\colon M_{c'}(\CC)\rightarrow M_c(\CC)$, which is the inverse of $\sigma_{c,c'}(\CC)$.

In order to show that $\sigma_{c,c'}(\CC)$ comes from a morphism $\sigma_{c,c'}\colon M_{c,\red}\rightarrow M_{c', \red}$ between the reduced moduli spaces, it suffices to show that we can do the above construction in family over a reduced base.

Thus, let $f\colon (\sX, \sB)\rightarrow V$ be a stable family in $M_c$ over a reduced base scheme $V$, where $\sB=c\lceil\sB\rceil$ and $\Supp\sB$ is $0$ if $c\geq \frac{6}{11}$ and a $\PP^1$-bundle over $V$ if $c<\frac{6}{11}$. Let $\sA_t$ be the union of the non-Gorenstein singularities of $\sX_t$ and the singularities of $\sX$ lying on $\sB_t$.  By Theorem~\ref{thm: liu2017minimal-volume}, the set $\sA_t$ consists of at most 3 points, and, as singularities of $\sX_t$, they are pairwise not isomorphic. Since $f$ is a stable family, $\sA:=\cup_{t}\sA_t$ consists of $\#\sA_t$ disjoint sections over $V$. Blowing up $\sX$ along $\sA$ sufficiently many times, we obtain a partial resolution $\widehat\Pi\colon\widehat \sX\rightarrow \sX$ such that its restriction to each fiber $\widehat\Pi_t\colon\widehat \sX_t\rightarrow \sX_t$, $t\in V$ satisfies the properties (i)--(iii) above.

We may write $\widehat\Pi_t^*(K_{\sX_t} + \sB_t) = K_{\widehat\sX_t}+\widehat\sB_t$ for some $\RR$-divisor $\widehat\sB_t$  with $\widehat\Pi_{t*}\widehat\sB_t=\sB_t$. Let $\widehat\sB_t=\sum_i b_i(c)\widehat \sB_{it}$ be the prime decomposition, where the coefficients $b_i(c)$ are independent of $t$. Then, up to reordering $\{\sB_{it}\}_i$, there is a cycle $\widehat\sB = \sum_i b_i(c)\widehat \sB_{i}$ on $\widehat\sX$ such that each $\widehat\sB_{i}$ is reduced, flat of relative codimension one over $V$, and $\widehat\sB_i\cap\widehat\sX_t=\widehat\sB_{it}$. By Theorem~\ref{thm: liu2017minimal-volume}, there are numbers $b_i(c')\leq c'$, independent of $t$, such that $K_{\widehat\sX_t}+\sum_i b_i(c')\widehat\sB_{it}$ is the positive part of $K_{\widehat\sX_t}+ c'\left\lceil\widehat\sB_t^{>0}\right\rceil$.

If $c'$ is rational, then all the $b_i(c')$ are rational, and we may take a sufficiently large and divisible $m$ such that $m(K_{\widehat\sX_t}+\sum_i b_i(c')\widehat\sB_{it})$ is Cartier and base point free, and thus defines the desired contraction $(\widehat\sX,c'\lceil\widehat\sB^{>0}\rceil) \rightarrow (\sX', \sB'=c'\lceil \sB'\rceil)$ onto the relative canonical model over $V$, using \cite[page 48, Corollary 2]{mumford2008abelian-varieties}. 

If $c'$ is irrational, then we may take a rational $c''$ that is sufficiently close to $c'$, and construct the canonical model $(\widehat \sX,c''\lceil\widehat\sB^{>0}\rceil)\rightarrow (\sX', c''\sB')$ over $V$. Since $c''$ is sufficiently close to $c'$, $(\sX', c'\sB')$ is then the canonical model for the original pair $(\widehat\sX, c'\lceil\widehat\sB^{>0}\rceil)$.

This map, sending the family $(\sX, \sB)\rightarrow V$ in $M_c$ over a reduced scheme to the family $(\sX',\sB')\rightarrow V$ in $M_{c'}$, induces a morphism $\sigma_{c,c'}\colon M_{c, \red}\rightarrow M_{c',\red}$, which restricts to $\sigma_{c,c'}(\CC)$ on the set of closed points. Since $\sigma_{c',c}$ is obviously the inverse morphism of $\sigma_{c,c'}$, we infer that  $\sigma_{c,c'}$ is an isomorphism.
\end{proof}

\begin{rmk}\label{rmk: M_c vs M_c'}
    The isomorphism $\sigma_{c,c'}\colon M_{c,\red}\rightarrow M_{c', \red}$ in the proof of Theorem~\ref{thm: same reduced} sends a K3 (resp.~rational) stable surface pair $(X,B)$ to a K3 (resp.~rational) stable surface pair $(X', B')$. In case $X$ and $X'$ are rational, their unique non-ADE singularities are isomorphic.
\end{rmk}

\begin{lem}\label{lem: elliptic fibration}
Let the notation be as in Theorem~\ref{thm: liu2017minimal-volume}. Set $$F=\Theta_1+2\Theta_2+3\Theta_3+4\Theta_4+5\Theta_5+6\Theta_6+4\Theta_7 + 2\Theta_8 + 3\Theta_9.$$
 Then $|F|$ has dimension $1$, defining a genus 1 fibration $h\colon Z\rightarrow \PP^1$ with $\Theta_0$ as a section. 
\end{lem}

\begin{proof}
For each $1\leq i\leq 9$, we have, 
\[
\sO_Z(F)|_{\Theta_i}\cong \sO_Z(K_Z+F)|_{\Theta_i}\cong \sO_{\Theta_i},
\]
and hence $\sO_F(F) \cong \sO_F$ and $H^0(F, \sO_F)\cong \CC$ by \cite[page 332, Lemma]{mumford1969eriques-classification}. By the long exact sequence of cohomology associated to the short exact sequence
\[
0\rightarrow \sO_Z \rightarrow \sO_Z(F) \rightarrow \sO_F(F) \rightarrow 0
\]
we infer that $\dim H^0(Z, \sO_Z(F))=2$, and there is a fibration $h\colon Z\rightarrow \PP^1$ with $F$ as a fiber. Since $\Theta_0\cdot F = \Theta_0\cdot \Theta_1=1$, $\Theta_0$ is a section of $h$.
\end{proof}

\begin{lem}\label{lem: X22 X211}
Let the notation be as in Theorem~\ref{thm: liu2017minimal-volume} and Lemma~\ref{lem: elliptic fibration}. 
Suppose that $Z$ is rational. Let $\varphi\colon Z\rightarrow \oZ$ be the relative minimal model over $\PP^1$, contracting the $(-1)$-curves in the fibers of $h$, and let $\bar h\colon\oZ\rightarrow\PP^1$ be the induced elliptic fibration:
\[
\begin{tikzcd}
    Z \arrow[rd, "h"']\arrow[rr, "\varphi"] &  & \oZ\arrow[ld, "\bar h"] \\
    & \PP^1 &
\end{tikzcd}
\]
Then the following holds.
\begin{enumerate}
    \item There is a $(-1)$-curve $G$ on $Z$ such that $B_Z^{\sss}+G\sim F$ and $\varphi$ is the contraction of $G$.
    \item $\oZ$ is isomorphic to one of the extremal rational elliptic surfaces $X_{22}$ or $X_{211}$, defined in \cite[Theorem~4.1]{miranda1986on-extremal-rational}.
    \item 
    The unique elliptic or cusp singularity of $X$ has degree $1$, that is, its exceptional divisor $B_{\tX}^\sss$ on the minimal resolution has self-intersection $(B_{\tX}^\sss)^{2}=-1$.
\end{enumerate}
\end{lem}
\begin{proof}
 Since $\Supp(B_Z^\sss)\cap\Supp(B_Z^{\ns})=\emptyset$ and $\Supp(B_Z^{\ns})$ contains a section $\Theta_0$, each connected component of $B_Z^\sss$ is properly contained in a fiber of $h$. Moreover, since the reduced divisor $B_Z^\sss$ is semistable in $|-K_Z|$, each of its connected component has arithmetic genus $1$. Let $B_{\oZ}^\sss:=\varphi(B_Z^\sss)\subset \oZ$. Then $K_{\oZ}+B_{\oZ}^\sss = \varphi_*(K_Z+B_Z^\sss)\sim 0$. By the canonical bundle formula for the rational elliptic surface $\bar h\colon \oZ\rightarrow\PP^1$, the curve $B_{\oZ}^\sss$ consists exactly of one fiber of type $I_k$ with $k\geq 0$. Since $K_Z+B_Z^\sss=\varphi^*(K_{\oZ}+B_{\oZ}^\sss)$, the blow-ups of $\varphi\colon Z\rightarrow\oZ$ occur over the smooth locus of $B_{\oZ}^\sss$. 
 
 Let $\overline{\Theta}_0=\varphi(\Theta_0)$. Then $\overline{\Theta}_0$ is a section of $\bar h$. Since $\Supp(B_Z^\sss)\cap \Supp(B_Z^\ns)=\emptyset$ and $\overline{\Theta}_0\cap B_{\oZ}^\sss\neq \emptyset$, there is at least one blow-up over $\overline{\Theta}_0\cap B_{\oZ}^\sss$. On the other hand, since $K_X+B$ is ample, its pullback $\pi^*(K_X+B)=K_{\tX}+B_{\tX}$ is positive on the curves intersecting $\Supp(B_{\tX})$, and so is $K_Z+B_Z$ on the curves intersecting $\Supp(B_Z)$. It follows that there are no further blow-ups other than the one at $\overline{\Theta}_0\cap B_{\oZ}^\sss$. Let $G$ be the $\varphi$-exceptional $(-1)$-curve. Then (i) holds.

Since $\bar h$ has one fiber of type $\II^*$ together with a section, it follows that $Z$ is an extremal rational elliptic surface. By the classification of such elliptic surfaces, we infer that $\oZ$ is isomorphic to $X_{22}$ or $X_{211}$ as in \cite[Theorem~4.1]{miranda1986on-extremal-rational}, the type depending 
on the other singular fibers 
which can only be once type $\II$ (for $X_{22}$)
or twice type $\I_1$ (for $X_{211}$; cf.\ the discussion in
Section \ref{ss:comp}).

(iii) is a consequence of (i).
\end{proof}

\section{The moduli stack \texorpdfstring{$M_c$}{} for \texorpdfstring{$0<c\leq \frac{7}{13}$}{}}\label{sec: M_c with c small}
In this section, we determine the moduli space $M_c$ as a stack for $0<c\leq \frac{7}{13}$. For this, we look at the log canonical ring $R(X, K_X+D)$ for $[(X, cD)]\in M_c$.
\begin{prop}\label{prop: canonical ring c small}
Resume Notation~\ref{nota: resolve X}. Take $c\leq \frac{7}{13}$ and $(X,B=cD)\in M_c$, where $D=\lceil B\rceil$ is isomorphic to $\PP^1$. For $n\geq 0$, let $P_n:=h^0(X, n(K_X+D))$ be the $n$-th plurigenera of $(X,D)$, and $\varphi_n\colon X\dashrightarrow\PP^{P_n-1}$ the rational map induced by the linear system $|n(K_X+D)|$. Then $(X, D)$ is a stable surface pair with $p_g(X, D)=1$ and $(K_X+D)^2=\frac{1}{42}$, and the following holds.
    \begin{enumerate}
      \item $\varphi_{14}$ is generically finite of degree 2 onto the image, and $\varphi_{21}$ is birational onto the image.
\item The generating series of the plurigenera satisfies
        \[
        \sum_{n\geq 0} P_nt^n=\frac{1-t^{42}}{(1-t)(1-t^{6})(1-t^{14})(1-t^{21})} 
        \]
        \item  There is an isomorphism of graded
        $\CC$-algebras
    \[
    R(X, K_X + D) = \bigoplus_{m\geq 0} H^0(X,m (K_X+D))\cong \CC[w,x,y,z]/(f)
    \]
    where $(\deg w,\,\deg x,\, \deg y,\, \deg z)=(1, 6, 14, 21)$, and 
    \[
    f(w,x,y,z)=f_g(w,x,y,z):=z^2+y^3+x^7+g(w,x,y)
    \]
    where $g$ is a nonzero element of the vector subspace $V_{42}\subset\CC[w,x,y]_{42}$ spanned by monomials of degree $42$ in $w,x,y$, not divisible by $y^2$ or $x^6$:
\[
w^4x^4 y,  w^{10} x^3y, w^{16} x^2 y,  w^{22} x y, w^{28}y, w^{12} x^5, w^{18}x^4, w^{24} x^3, w^{30} x^2,  w^{36}x, w^{42}.
\]
    \end{enumerate}
\end{prop}

\begin{proof}
By Theorem~\ref{thm: liu2017minimal-volume} (iii), the three arms of the dual graph of $\lceil B_{\tX}^{\ns}\rceil$ contract to three ADE singularities, of type $\frac{1}{2}(1,1)$, $\frac{1}{3}(1,2)$, and $\frac{1}{7}(1,6)$ respectively. Thus $X$ is again a Gorenstein surface with $K_X\sim 0$, and the curve $D=\pi(\Theta_6)$ passes through the aforementioned three singularities. We have 
    \begin{equation}\label{eq: pullback D}
    \pi^*D = \Theta_6 + \frac{1}{2}\Theta_9+\frac{2}{3} \Theta_7 + \frac{1}{3} \Theta_8 + \sum_{j=0}^5\frac{j+1}{7}\Theta_j        
    \end{equation}
Clearly, $(X,D)$ has log canonical singularities, $K_X+D=D$ is ample, and the volume is $(K_X+D)^2=D^2=\frac{1}{42}$.

\medskip

(i)  Let $\pi\colon \tX = Z\rightarrow X$ be the minimal resolution, as in Theorem~\ref{thm: liu2017minimal-volume}. Then, for $n\geq 0$, there is a natural identification
\[
|n(K_X+D)| = |nD|\cong |\lfloor n\pi^*D \rfloor|
\]
By \eqref{eq: pullback D}, there is a $\tD_{6}\in |\lfloor 6\pi^*D\rfloor|$ containing the $\II^*$-fiber of $h\colon Z\rightarrow \PP^1$ as a subdivisor (cf.~Lemma~\ref{lem: elliptic fibration}):
\[
F= \Theta_1+2\Theta_2+3\Theta_3+4\Theta_4+5\Theta_5+6\Theta_6+4\Theta_7+2\Theta_8+3\Theta_9
\]
 and a $\tD_{14}\in |14\pi^*D|$ (resp.~$\tD_{21}\in |21\pi^*D|$) containing the divisor $2F+2\Theta_0$ (resp.~$3F+3\Theta_0$), where we recall that $\Theta_0$ is the section of the elliptic fibration $h\colon Z\rightarrow \PP^1$ (see Theorem~\ref{thm: liu2017minimal-volume} and Lemma~\ref{lem: elliptic fibration} for notation). 

For $m\in \{2,3\}$, consider the long exact sequence of cohomology associated to the short exact sequence 
 \[
 0\rightarrow \sO_{Z}((m-1)\Theta_0+ mF)\rightarrow \sO_{Z}(m\Theta_0+mF)\rightarrow \sO_{F}(m\Theta_0+mF)\rightarrow 0.
 \]
The Leray spectral sequence gives 
\begin{multline*}
H^1(Z, \sO_{Z}((m-1)\Theta_0+ mF)) =H^1(\PP^1, h_*\sO_{Z}((m-1)\Theta_0+ mF))\\
=
\begin{cases}
 H^1(\PP^1,\sO_{\PP^1}(2))=0 & \text{if $m=2$} \\
 H^1(\PP^1,\sO_{\PP^1}(-1))=0 & \text{if $m=3$}
\end{cases}  
\end{multline*}
where the second equality follows from \cite[II.4.3]{miranda1989the-basic-theory}. Thus we have a surjection
\[
H^0(Z, m\Theta_0+mF)\rightarrow H^0(F,\sO_{F}(m\Theta_0+mF))\rightarrow 0.
\]
It follows that the restriction of $|m\Theta_0+mF|$ to $F$ is a complete linear system of degree $m$, and hence induces an embedding (resp.~$2$-to-$1$ map) of $F$ if $m=3$ (resp.~$m=2$). On the other hand, since $2\Theta_0+2F\geq 2F$ and $|2F|$ separates fibers of $h$, the linear system $|2\Theta_0+2F|$ also separates fibers of $h$. In conclusion, $|m\Theta_0+mF|$ induces a birational (resp.~degree $2$) map from $Z$ if $m=3$ (resp.~$m=2$). 

Correspondingly, the pluricanonical map $\varphi_{21}$ of $X$ is birational onto the image,
and $\varphi_{14}$
is generically finite of degree $2$ onto the image.

\medskip

(ii) We first compute directly the correction terms $\delta_{x}(nD)$ for $1\leq n\leq 3$, 
using \eqref{eq: pullback D} and Theorem~\ref{thm: Blache}.
They are listed as follows:
    \begin{center}
    \begin{tabular}{|c||c|c|c|}
    \hline
           & $\frac{1}{2}(1,1)$ &$\frac{1}{3}(1,2)$ &$\frac{1}{7}(1,6)$\\ \hline\hline
       $\delta_x(D)$  & $-\frac{1}{4}$ & $ -\frac{1}{3}$& $-\frac{3}{7}$\\ \hline
       $\delta_x(2D)$ & 0 & $-\frac{1}{3}$ & $-\frac{5}{7}$ \\ \hline
       $\delta_x(3D)$  & $-\frac{1}{4}$ & $0$ & $-\frac{6}{7}$ \\ \hline
    \end{tabular}
    \end{center}
For the other $n$, we may deduce $\delta_x(nD)$ from the above table by Lemma~\ref{lem: delta D}. Plugging these corrections into the Riemann--Roch formula for $\chi(X, \sO_X(nD))$, we obtain $h^0(nD)$ for $1\leq n\leq 45$. 
\begin{center}
\begin{tabular}{|c||c|c|c|c|c|c|c|c|}
\hline
$n$& $1$--$5$ & $6$--$11$ & $12$--$13$ & $14$--$17$ & $18$--$19$ & $20$ & $21$--$23$  & $24$--$25$    \\ \hline
$h^0(nD)$ & $1$ & $2$ & $3$ & $4$ & $5$ & $6$ & $7$ & $8$\\ \hline\hline
$n$ &
$26$ & $27$ & $28$--$29$ & $30$--$31$ & $32$ & $33$ & $34$ & $35$
\\
\hline
$h^0(nD)$ & $9$ & $10$ & $11$ & $12$ & $13$ & $14$ & $15$ & $16$
\\
\hline
\hline
$n$ &
$36$--$37$ & $38$ & $39$ & $40$ & $41$ & $42$--$43$ & $44$ & $45$
\\
\hline
$h^0(nD)$ &
$17$ & $18$ & $19$ & $20$ & $21$ & $23$ & $24$ & $25$
\\
\hline
\end{tabular}
\end{center}
Set 
\[
A(t)=\sum_{n\geq 0}a_nt^n:=(\sum_{n\geq 0} P_n t^n)(1-t)(1-t^{6})(1-t^{14})(1-t^{21})
\]
We need to prove that $A(t)=1-t^{42}$. Using the values of $P_n$ for $n\leq 42$ computed above, one may check directly that  
\[
A(t)= 1- t^{42} + \text{higher order terms}.
\]
For $n>42$, the coefficient of $t^n$ in $A(t)=\sum_{n\geq 0}a_nt^n$ is
\begin{multline*}
 a_n=(P_n-P_{n-1}) - (P_{n-6}-P_{n-7}) - (P_{n-14}-P_{n-15})  
  + (P_{n-20}-P_{n-21}) \\
 - (P_{n-21}-P_{n-22})
 + (P_{n-27}-P_{n-28}) + (P_{n-35}-P_{n-36}) - (P_{n-41}-P_{n-42})   
\end{multline*}
For each of the three singularities $x$ of $X$ lying on $D$, denoting by $k_x$ the Cartier index of $D$ at $x$,  we can group the $P_m$'s appearing in the above expression of $a_n$ into pairs of the form $\pm(P_m - P_{m-k_x})$. The correction terms $\delta_x(mD)$ and $\delta_x((m-k_x)D)$ in the Blache--Riemann--Roch formula for $P_m$ and $P_{m-k_x}$ are the same, and hence cancel each other in $P_m - P_{m-k_x}$.
To compute the above sum, we may thus use Blache's Riemann--Roch formula \eqref{eq: Blache}, with the correction term omitted, and easily conclude that 
$a_n=0$ for $n> 42$. Hence $A(t)=1-t^{42}$, as required.

\medskip

(iii) First, there are generators $w,x,y,z$ in degrees $1,6,14,21$ respectively, satisfying a relation $f(w,x,y,z)=0$ in degree $42$ (which thus defines $X$ in weighted projective space $\mathbb P(1,6,14,21)$). 
Since $\varphi_{21}$ is birational while $\varphi_{14}$ is not, the coefficient of $z^2$ in $f$ is nonzero. Completing the square with respect to $z$, we may assume that $f$ does not contain any monomial that is linear in $z$: 
\[
 f=z^2+g_1(w,x,y)
\]
where $g_1\in \CC[w,x,y]_{42}$. 
\begin{claim}\label{claim: y3 x7 appear}
The coefficients of the terms $y^3$ and $x^7$ in $f$ must be nonzero. 
\end{claim}
In fact, if $x^7$ does not appear in $f$, then $X=(f=0)$ passes through the point $(0:1:0:0)$, which is the quotient of the singularity 
\[
(0,0,0)\in Y_x=(f(w,1,y,z)=0)\subset \CC_{w,y,z}^3
\]
under the action of $\mu_6\colon (w,y,z)\mapsto (\xi_6 w, \xi_6^{14}y, \xi_6^{21}z)$ for a primitive $6$-th root $\xi_6$ of $1$. Since $f(w,1,y,z)=z^2+ay^3+\text{terms of order $\geq 8$}$, one sees easily that $Y$ is not log canonical at the origin (cf.~\cite[Proposition~III.3.2]{miranda1989the-basic-theory}). Now the action of $\mu_6$ only fixes $(0,0,0)$, we have $K_Y=\mu^*K_X$, where $\mu\colon Y\rightarrow X\cap (x\neq 0)$ the quotient map. We infer that $X$ is not log canonical at $(0:1:0:0)$ by \cite[Proposition~5.20]{kollar1998birational-geometry}, a contradiction. In a similar way, we infer that the coefficient of $y^3$ in $f$ is nonzero.

Up to rescaling the variables $x$ and $y$ and completing the cube in $y$ and the seventh power in $x$, we may now write 
\[
f=z^2+y^3+x^7+g(w,x,y),
\]
where $g\in \CC[w,x,y]_{42}$ does not contain monomials divisible by $y^2$ or $x^6$. Note that $g\neq 0$, since $(z^2+y^3+x^7=0)\subset \PP(1,6,14,21)$ has a non-lc singularity at $P_w=(1:0:0:0)$.
\end{proof}

\begin{rmk}
The K3 surfaces with ADE singularities of degree $42$ in $\PP(1,6,14,21)$ appear as No.~88 in  Reid's list of K3 hypersurfaces (\cite[p.~140]{fletcher2000working-with-weighted}; alternatively No.~14 in \cite[Table 4.6]{yonemura1990hypersurface-simple}). In view of Theorem~\ref{thm: liu2017minimal-volume} and Proposition~\ref{prop: canonical ring c small}, there are also Gorenstein rational surfaces of degree $42$ in $\PP(1,6,14,21)$ with exactly one simple elliptic or cusp singularity.
\end{rmk}
We make now an analysis of singularities appearing in a hypersurface $X_g:=(f_g=0)\subset \PP(1,6,14,21)$ for any $g\in V_{42}\setminus\{0\}$. Note that the affine part 
\[
X_g\cap (w\neq 0) = (f_g(1,x,y,z)=0)\subset \CC^3
\]
can be viewed as a deformation of the exceptional unimodular singularity 
\begin{equation}\label{eq:S237sing}
S_{237}:=(z^2+y^3+x^7=0)\subset \CC^3.
\end{equation}
By unfolding the equation of $S_{237}$, Brieskorn \cite{Bri79} showed that the possible non-canonical singularities appearing on $X_g\cap (w\neq 0)$ can only be singularities with normal form
\[
(z^r+y^s+x^t+cxyz=0)\subset \CC^3, \quad c\neq 0,
\]
with $(r,s,t)=(2,3,7)$ or $(2,3,6)$; in the terminology of modern birational geometry, they are log canonical hypersurface singularities; more precisely, it is a simple elliptic singularity (resp.~cusp singularity) if $\frac{1}{r}+\frac{1}{s}+\frac{1}{t}=1$ (resp.~$<1$). The precise locus in the semiuniversal deformation space $\Def(S_{237})$ corresponding to non-canonical singularities is given by Hashimoto and Ueda \cite[Corollary~3.2]{hashimoto2022ring-of-modular}. On the other hand, Looijenga \cite{looijenga1984smoothing-components} studied the smoothing of so-called Fuchsian singularities, which include $S_{237}$ as a special case, by embedding them into a projective normal surface with trivial canonical class (for example, $S_{237}\subset \oS_{237}:=(z^2+y^3+x^7=0)\subset \PP(1,6,14,21)$) and then smoothing the projective surface; the minimal resolution of singularities appearing along the curve at infinity can be found in \cite[(1.6)]{looijenga1984smoothing-components}. The following proposition collects the known results on the singularities of $X_g$ from \cite{Bri79, looijenga1984smoothing-components, hashimoto2022ring-of-modular}, both on the affine part $X_g\cap (w\neq 0) $ and along the curve at infinity $X_g\cap (w= 0)$, with refinements made for the purpose of this paper; we give complete proofs when appropriate references are lacking.
\begin{prop}\label{prop: singularity analysis c small}
Let $\CC[w,x,y]$ be the polynomial ring with $(\deg w, \deg x, \deg y)=(1,6,14)$, and $V_{42}\subset \CC[w,x,y]_{42}$ the subspace as in Proposition~\ref{prop: canonical ring c small} (iii). Then for any $g,g'\in V_{42}\setminus \{0\}$, the following holds.
\begin{enumerate}
    \item Along the divisor $D_g:=X_g\cap (w=0)$ at infinity, which is isomorphic to $\PP^1$, the surface $X_g$ has an $A_1$-singularity at $(0:1:-1:0)$, an $A_2$-singularity at $(0:-1:0:1)$, and an $A_6$-singularity at $(0:0:-1:1)$, so that the inverse image of $D_g$ on the minimal resolution $\pi\colon \tX_g\rightarrow X_g$ is a $T_{2,3,7}$-configuration of $(-2)$-curves.
    \item The affine open part $X_g\cap (w\neq 0)$ of $X_g$ contains a non-canonical singularity if and only if there is $(a,b)\in \CC^2\setminus\{(0,0)\}$ such that $g= g_{a,b}$, where $g_{a,b}= a\ell^4w^4y + \ell^6(x+6bw^6) -x^7$, with $\ell=x-bw^6$.  
    \item For $(a,b)\in \CC^2\setminus\{(0,0)\}$, the following holds.
    \begin{enumerate}
        \item $X_{g_{a,b}}$ has a unique non-canonical (hypersurface) singularity at $(1:b:0:0)$, which is simple elliptic (resp.~cusp) if $b\neq 0$ (resp.~$b=0$).
        \item There is a genus one fibration $\tilde h\colon \tX_{g_{a,b}}\rightarrow \PP^1$ such that the terminal component $\widetilde\Theta_0$ of the longest arm of the $T_{2,3,7}$-configuration is a section of the fibration $\tilde h$. Moreover, the relative minimal model $\oX_{g_{a,b}}\rightarrow \PP^1$ is  $X_{22}$ (resp.~$X_{211}$) if $a=0$ (resp.~$a\neq 0$), obtained by blowing down one vertical $(-1)$-curve.
        \item Let $\tF_0$ be the fiber of $\tilde h$ containing the vertical $(-1)$-curve of (b). Then the image $F_0\subset X_{g_{a,b}}$ of $\tF_0$ is the curve defined by $\ell=0$, and the $j$-invariant of $\tF_0$
        is $12^3\cdot\frac{4a^3 }{4a^3- 1323b^2}$.
    \end{enumerate}
    \item   $X_g$ has only log canonical hypersurface singularities, $K_{X_g}\sim 0$, and for any $0<c\leq 1$, $(X_g, cD_g)$ is a stable surface pair with volume $\frac{1}{42}c^2$.
    \item $X_g\cong X_{g'}$ if and only if there is an $\lambda\in \CC^*$ such that
    \[
    g(w,x,y)=(\lambda.g')(w,x,y):=g'(\lambda w,x,y).
    \]
\end{enumerate}
\end{prop}
\begin{proof}
(i) see the proof of \cite[Proposition~3.4]{hashimoto2022ring-of-modular}, which works for any $X_g$.

\medskip

(ii) is essentially \cite[Corollary~3.2]{hashimoto2022ring-of-modular}; we spell out the details of its proof here. The affine surface $X_g\cap (w\neq 0)\subset \CC^3$ is defined by 
\[
f_g(1,x,y,z) = z^2+y^3+ x^7 + g(1,x,y).
\]
We may write 
\begin{equation}\label{eq: g}
    x^7 + g(1,x,y) = A(x)y+ B(x),
\end{equation}
where $A, B$ are appropriate polynomials in $x$:
\begin{equation}\label{eq: A B}
A(x)=ax^4+\text{terms of degree $\leq 3$},\quad  B(x) =x^7 +\text{terms of degree $\leq 5$}.
\end{equation}
By \cite[III.3.2]{miranda1989the-basic-theory}, non-canonical singularities occur if and only if there is a linear polynomial $\bar\ell\in \CC[x]$ such that $\bar\ell^4\mid A$, $\bar\ell^6\mid B$. In view of \eqref{eq: A B}, we may choose $\bar \ell=x-b$, and then
\begin{equation}\label{eq:ABDelta}
   A = a\bar \ell^4, \quad B = \bar \ell^6(x+6b)
\end{equation}
Homogenizing \eqref{eq: A B} with $(\deg w, \deg x, \deg y)=(1,6,14)$, we recover $g$ as asserted:
\[
g(w,x,y) = w^{42}A(x/w^6) + w^{28}B(x/w^6)y -x^7 = a\ell^4w^4y + \ell^6(x+6bw^6) -x^7
\]
where $\ell=x-bw^6$ is the homogenization of $\bar\ell$. 

\medskip

(iii) (a) The non-canonical singularity occurs at $(\bar\ell=x=y=0)$, which consists of the single point $P=(1:b:0:0)$. If $b\neq 0$, then $\bar \ell^6$ is the maximal power of $\bar\ell$ dividing $B$, and one can verify directly that $P\in X_{g_{a,b}}$ is a simple elliptic hypersurface singularity by resolving it; if $b=0$, then $\bar \ell^7\mid B$, and $P\in X_{g_{a,b}}$ is a cusp hypersurface singularity.

\medskip

(iii) (b) Note that $\tX_{g_{a,b}}$ is a projective surface with the same structural description as $Z$ of Lemmas~\ref{lem: elliptic fibration} and \ref{lem: X22 X211}, and the same proofs of these lemmas apply, with $Z$ (resp.~$\oZ$) there replaced with $\tX_{g_{a,b}}$ (resp.~$\oX_{g_{a,b}}$).

As in the proof of (ii), the affine equation of $X_{g_{a,b}}$ is
\[
z^2+y^3+ A(x)y + B(x) = 0 
\]
where $A = a\bar \ell^4, \quad B = \bar \ell^6(x+6b)$. Then the affine equation of $\oX_{g_{a,b}}$ is (\cite[III.3.4]{miranda1989the-basic-theory})
\[
z^2+y^3+ \oA(x)y + \oB(x) = 0 
\]
where  $ \oA = A/\bar\ell^4=a, \quad \oB =B/\bar\ell^6= x+6b$. The discriminant of $\oX_{g_{a,b}}$ is then
\begin{equation}\label{eq:hatABDelta}
   \Delta =4\oA^3-27\oB^2= 4a^3 - 27(x+ 6b)^2.
\end{equation}
and the $j$-map $j\colon \A^1_x\rightarrow \PP_j^1$ is
\begin{equation}\label{eq: j map}
j(x) =12^3\cdot \frac{4\oA^3}{4\oA^3-27\oB^2} =12^3\cdot\frac{4a^3}{4a^3 - 27(x+ 6b)^2}.
\end{equation}
If $a=0$ then $j(x)\equiv 0$ and $\oX_{g_{a,b}}$ has a $\II$-fiber over $x=-6b$. If $a\ne0$ then $\oX_{g_{a,b}}$ has two $\I_1$-fibers, one over each of the two solutions of $\Delta(x) =0$. In either case it also has the $\II^*$-fiber at infinity, i.e., the one contained in the image of $(w=0)\cap \tX_{g_{a,b}}$. It follows that for $a=0$ (resp. $a\ne0$), $\oX_{g_{a,b}}$ is the extremal rational elliptic surface $X_{22}$ (resp. $X_{211}$) of Miranda--Persson \cite{miranda1986on-extremal-rational}.

\medskip

(iii) (c) The curve $F_0$ is the curve containing the unique non-canonical singularity of $X_{g_{a,b}}$, so defined by $\bar\ell=0$, i.e., $x=b$. By \eqref{eq: j map}, the $j$-invariant of this fiber is
\[
j(b) =12^3\cdot \frac{4a^3 }{4a^3 -27(7b)^2} =12^3\cdot\frac{4a^3 }{4a^3 -1323b^2}  .
\]

\medskip

(iv) It follows from (i)--(iii) that $X_g$ has only log canonical hypersurface singularities and  $D_g$ is an ample $\QQ$-Cartier divisor such that $(X_g, D_g)$ is log canonical and 
\[
D_g^2 = (\pi^*D_g)^2= -2+\frac{1}{2} + \frac{2}{3} +\frac{6}{7} = \frac{1}{42}.
\]
Thus, $(X_g, c D_g)$ is also log canonical for any $0<c\leq 1$. By adjunction, we have $K_{X_g}\sim 0$, and hence $K_{X_g}+cD_g \sim cD_g$ is ample. It follows that $(X_g, c D_g)$ is a stable surface pair with volume 
\[
(K_{X_g}+cD_g)^2 = c^2 D_g^2 = \frac{1}{42}c^2
\]

(v) By (iv) the stable surface pair $(X_g, cD_g)$ achieves the minimal volume $\frac{1}{42}c^2$ as in Theorem~\ref{thm: liu2017minimal-volume} (1). The computation of Proposition~\ref{prop: canonical ring c small} shows that the inclusion $X_g \subset \PP(1,6,14,21)$ is via the identification $R(X_g, D_g)\cong \CC[w,x,y,z]/(f_g)$. The same holds for $X_{g'}$. Thus, $X_g\cong X_{g'}$ if and only if there is a coordinate change $\sigma$ of $\PP(1,6,14,21)$ such that $f_g(w,x,y,z) = f_{g'}\circ \sigma(w,x,y,z)$. Since $f_g$ and $f_{g'}$ share the same principal part $z^2+y^3+x^7$ and do not contain other monomials divisible by one of $z$, $y^2$ and $x^6$, the only possibility is $\sigma(w,x,y,z) = (\lambda_1 w, \lambda_2 x,\lambda_3 y,\lambda_4z)$, where $\lambda_i\in \CC^*$ for $1\leq i\leq 4$, and $\lambda_2^7=\lambda_3^{3}=\lambda_4^{2}=1$. We may find an $\alpha\in \CC^*$ such that $\alpha^6=\lambda_2^{-1}$ and $\alpha^{14}=\lambda_3^{-1}$. Since $(\deg w, \deg x, \deg y)=(1,6,14)$ and $g$ is homogeneous of degree $42$, taking $\lambda =\lambda_1\alpha$, we have
\[
g'(\lambda_1w,\lambda_2x,\lambda_3 y) =g'(\lambda_1\alpha w,\lambda_2\alpha^6x,\lambda_3 \alpha^{14} y)= g'(\lambda w, x,y)
\]
so $g(w,x,y)=g'(\lambda w,x,y)$ holds, as required.
\end{proof}

Recall  the vector subspace $V_{42}\subset \CC[w,x,y]_{42}$ with $(\deg w, \deg x, \deg y)=(1,6,14)$, spanned by the following monomials
\[
w^4x^4 y,  w^{10} x^3y, w^{16} x^2 y,  w^{22} x y, w^{28}y, w^{12} x^5, w^{18}x^4, w^{24} x^3, w^{30} x^2,  w^{36}x, w^{42}.
\]
Denote $V_{42}^*:=V_{42}\setminus\{0\}$. By Proposition~\ref{prop: singularity analysis c small} (iv), there is a flat family of stable surface pairs 
\[
\pi_{V_{42}^*}\colon (\sX_{V_{42}^*}, c\sD_{V_{42}^*})\subset \PP(1,6,14,21)\times V_{42}^*\rightarrow V_{42}^*
\]
such that over any $g\in V_{42}^*$, the fiber $(\sX_g, c\sD_g)$ is the stable surface pair $(X_g, cD_g)$, where $X_g:=(f_g=0)\subset \PP(1,6,14,21)$, with $f_g:=z^2+y^3+x^7+g$, and  $D_g=X_g\cap (w=0)$. By Proposition~\ref{prop: singularity analysis c small} (v), this family descends to a family $\pi_{\PP}\colon (\cX_\bP, c\cD_\bP)\to\bP$, where $\bP=V_{42}^*/\CC^*$ and the action of $\CC^*$ on $V_{42}^*$ has weights equal to the exponent of $w$ in the basis monomials of $V_{42}$. The family $\sX_\bP\to\bP$ can be defined by the equation $f(w,x,y,z,t_i)=0$, where
\begin{equation}\label{eq:brieskorn}
 f = z^2 + y^3 + A_{28}(w,x)y + B_{42}(w,x) \quad\text{with}
\end{equation}
 \begin{align*}
 A &= A_{28} = t_4 x^4 w^4  + t_{10} x^3 w^{10} + t_{16} x^2 w^{16} + t_{22} xw^{22} + t_{28} w^{28}\\
B &=B_{42}= x^7 + t_{12} x^5 w^{12} + t_{18} x^4 w^{18} + t_{24} x^3 w^{24} + t_{30} x^2 w^{30} + t_{36} xw^{36} + t_{42} w^{42}
\end{align*}
We have thus
\[
\PP = \PP(4,10,12,16,18,22,24,28,30,36,42)
\]
as stacks; the evenness of every weight reflects the fact that the generic automorphism group of $(\sX_t, \sD_t)$ is $\mu_2$, generated by the involution $z\mapsto -z$. 
\begin{defn}
The family $\pi_{\PP}\colon (\cX_\bP, c\cD_\bP)\to\bP$ is called the \emph{Brieskorn family}.
\end{defn}

Now we may determine the moduli space $M_c$ as a stack for $0<c\leq \frac{7}{13}$.
\begin{thm}\label{thm:PtoMc}
    For any $0<c\le \frac7{13}$ there is an isomorphism of Deligne--Mumford stacks $$\bP(4, 10, 12, 16, 18, 22, 24, 28, 30, 36, 42)\to M_c$$ 
    and of their coarse moduli spaces.
\end{thm}
\begin{proof}
The classifying morphism $\PP\rightarrow M_c$ is a morphism between Deligne-Mumford stacks, and it is bijective on points of their coarse moduli spaces by Proposition~\ref{prop: singularity analysis c small}.

Vice versa, we claim that any family $\pi\colon (\cX,c\cD)\to S$ in $M_c$ over a Noetherian base scheme $S$ for $c\le\frac7{13}$ is obtained by pullback from Brieskorn's family by a morphism $S\to\bP$, thus giving the inverse map $M_c\to\bP$.
    
Indeed, for any fiber $(\cX_s,c\cD_s)$ over a point $s\in S$, by Proposition~\ref{prop: canonical ring c small} the algebra $R(\cX_s,K_{\cX_s}+\cD_s)$ is isomorphic to $\bC[w,x,y,z]/(f)$ for some generators $w,x,y,z$ of $H^0(m(K_{\cX_s}+\cD_s))$ with $m=1,6,14,21$ and a degree~$42$ polynomial $f = x^7+ y^3+ z^2 + \dotsc\in \bC[w,x,y,z]$. 

    For any pair $(X,cD)$ in $M_c$ one has $H^i(m(K_X+D))=0$ for all $m\ge0$ and $i>0$ by \cite[Theorem~1.7]{fujino2014fundamental-theorems}. Then by Cohomology and Base Change theorem,   
    \[
    R(\cX/S, K_{\cX/S} + \cD) = \oplus_{m\ge0} \pi_* \cO_\cX( m(K_{\cX/S} + \cD) )
    \]
    is a locally free  $\cO_S$-algebra with graded pieces of finite rank. Extending locally the generators of $R(\cX_s,K_{\cX_s}+\cD_s)$ to $R(\cX/S, K_{\cX/S} + \cD)$ by Nakayama's lemma, we get an open affine neighborhood $U=\Spec \cR\ni s$ and a short exact sequence
    \[
    0 \to K \to \cR[w,x,y,z]_{42}\simeq \cR^{24}\to Q=H^0(\cX_U, 42(K_{\cX_U/U} + D_U)) \to 0,
    \]
    where $Q$ and $K$ are projective $\cR$-modules, with $\rank Q=23$ and $\rank K =1$. Shrinking $U$ further, we can assume that $K=\cR f$, where $f = \alpha x^7+\beta y^3+\gamma z^2 + \dotsc\in \cR[w,x,y,z]$ is a degree~$42$ polynomial  with $\alpha,\beta,\gamma$ nowhere vanishing on $U$. Then for the restricted family $\cX_U\to U$ one has
    \[
    R(\cX_U/U, K_{\cX_U} + D_U) = \cR[w,x,y,z]/(f),
    \]
    so $\cX_U = \Proj R(\cX_U/U, K_{\cX_U} + D_U)$ is the hypersurface $Z(f)\subset \bP(1,6,14,21)\times U$.
    Let
    \[
    \cR' = \frac{\cR[\alpha',\beta',\gamma']}{(\alpha'{}^7-\alpha, \beta'{}^3-\beta, \gamma'{}^2-\gamma)}, 
    \qquad U' = \Spec \cR'.
    \]

    Then $U'\to U$ is 
    a Galois $(\mu_7\times\mu_3\times\mu_2)$-cover and 
    the pullback family $\cX_{U'}\subset \bP(1,6,14,21)\times U'$ is defined by a polynomial $f' = x^7+y^3+z^2 + \dotsc \in \cR'[w,x,y,z]$. Completing the square in $z$, cube in $y$ and the seventh power in $x$, we put $f'$ in a standard form that does not contain $z$, $y^2$ and $x^6$. 
     Let us write
    \begin{equation}\label{eq:general-family}
    f' = z^2 + y^3 + A_{28}(w,x)y + B_{42}(w,x), \quad\text{with}
    \end{equation}
    \begin{align*}
    A &= A_{28} = r'_4 x^4 w^4  + r'_{10} x^3 w^{10} + r'_{16} x^2 w^{16} + r'_{22} xw^{22} + r'_{28} w^{28}\\
    B &=B_{42}= x^7 + r'_{12} x^5 w^{12} + r'_{18} x^4 w^{18} + r'_{24} x^3 w^{24} + r'_{30} x^2 w^{30} + r'_{36} xw^{36} + r'_{42} w^{42}
    \end{align*}
    for some $r'_i\in \cR'$. For every point $s'\in U'$ at least one of the coefficients $r'_i(s')\in\bC$ evaluated at $s'$ is nonzero. Indeed, the surface defined by  Equation~\eqref{eq:S237sing} is not log canonical but $\cX_{s'}$ is. 
    Comparing Equations \eqref{eq:general-family} and \eqref{eq:brieskorn}, we see that 
    the correspondence $(t_i)\to (r'_i)$ defines a morphism $\varphi'\colon U'\to \bP$ for which $(\cX_{U'},\cD_{U'})$ is the pullback from the Brieskorn's family $(\cX_\bP,\cD_\bP)\to\bP$.     

    The standard equation $f'$ is unique up to rescaling $x\to\xi_7x$, $y\to\xi_3y$, $z\to\xi_2z$ by some $p$-th roots of unity $\xi_p$ and $w\to\lambda w$ with invertible $\lambda\in \cR'$. Any such rescaling defines the same map $\varphi'\colon U'\to \bP$. Since $\varphi'$ commutes with the $(\mu_7\times\mu_3\times\mu_2)$-action, it descends to a morphism $\varphi\colon U\to\bP$ for which $(\cX_U,\cD_U)$ is isomorphic to the pullback of  $(\cX_\bP,\cD_\bP)\to\bP$. And since $\varphi$ was obtained in a canonical way, for different choices of neighborhoods $U\ni s$ of points $s\in S$ the morphisms $U\to \bP$ canonically glue to a morphism $S\to\bP$ for which
    $\pi\colon (\cX,\cD)\to S$ is the pullback from the Brieskorn's family. 
\end{proof}



\begin{rmk}\label{rmk: dim strata}
By the description of the Brieskorn family given in this section, the stratum $M_c^\Kth=M_c^\I$ of the stratification \eqref{eq: stratification} is an open dense subscheme of $M_c$, while $\dim M_c^\II=1$ and $M_c^\III$ consists of a single point. Proposition~\ref{prop: singularity analysis c small} induces an embedding of moduli stacks
\[
\PP(4,6)\rightarrow \PP,\quad (a:b)\mapsto [g_{a,b}]
\]
whose image as a set is exactly $M_c^\II\cup M_c^\III$. Of course $\bP(4,6)$ is the moduli stack $\oM_{1,1}$ of genus~$1$ curves with a point, whose coarse moduli space is the modular curve $X(1)=\bP^1_j$, the $j$-line. See also Theorem~\ref{thm:stable-limits}.
\end{rmk}

\section{Description of \texorpdfstring{$M_c$}{} for \texorpdfstring{$0<c\leq \frac{7}{13}$}{} using the KSBA theory}
\label{sec:proof-by-degenerations}
In the papers of Pinkham \cite{Pin77}, Brieskorn \cite{brieskorn1981unfolding-of-exceptional}, and Looijenga \cite{looijenga1983smoothing-components, looijenga1984smoothing-components}, the authors related the semiuniversal deformation of quasihomogeneous hypersurface surface singularities to the moduli space of K3 surfaces, and one may also use the degenerations of K3 surfaces to understand the degeneration of singularities, as suggested by  Mumford (\cite[Footnote on page 89]{brieskorn1981unfolding-of-exceptional}). 

For a stable surface pair $(X, cD)$ in $M_c$, $0<c\leq \frac{7}{13}$, the surface $X$ has trivial canonical class, and hence its minimal resolution is a K3 surface if $X$ has at most canonical singularities. As in Remark~\ref{rmk: dim strata}, such K3 stable surface pairs form an open dense subscheme $M_c^\Kth$ of $M_c$. Other stable surface pairs lying in the strata $M_c^\II$ and $M_c^\III$ can be viewed as degenerations of K3 stable surface pairs. 

In this section, using the theory of KSBA compactifications of moduli of K3 surfaces developed in \cite{AE23} and the description of Kulikov models of elliptic K3 surfaces  in \cite{ABE22}, we may give a satisfactory description of the degeneration of K3 surfaces involved in the compactification $M_c^\Kth \subset M_c$. In doing so,  we describe nef Kulikov models for the families of K3 surfaces in $F_\Lattice$ and prove that $F_\Lattice$ admits a canonical polarizing divisor $R$ which is recognizable in the sense of \cite{AE23}. The conclusion is that $\oF_\Lattice^R\cong \oF_\Lattice^{\rBB}$; see Corollary~\ref{cor:oFR=oFBB}.

\subsection{Generalities on moduli of (quasi)polarized K3 surfaces}\label{sec:generalities_lattice_moduli}

We briefly recall the main facts about the moduli of lattice (quasi)polarized K3 surfaces.
The standard references here are \cite{dolgachev1996mirror-symmetry,dolgachev2007moduli-of-k3}. Unfortunately, for a general lattice $\Lattice$ there is an issue; it was corrected in \cite[Section~2.1]{AE23} and \cite{AE25}. Fortunately, this issue plays no role if the lattice $\Lattice$ is unimodular, which is the case for our lattice $\Lattice=U\oplus E_8$.

Let $L_{K3}=\II_{3,19}\simeq U^{\oplus 3}\oplus E_8^{\oplus 2}$ be the $K3$ lattice. To any primitive hyperbolic sublattice $\Lattice \subset L_{K3}$ one associates a Type IV symmetric Hermitian domain
\[
\DD_\Lattice:=\PP\{x\in \Lattice^{\perp}\otimes \CC\mid x\cdot x=0, x\cdot\bar x>0\},
\]
and the discrete arithmetic group $\Gamma$ acting on it, consisting of the isometries $\gamma\in O(\Lattice^\perp)$ that can be extended to $O(L_{K3})$. The quasi-projective variety $\DD_\Lattice/\Gamma$ is known as the moduli space of $\Lattice$-(quasi)polarized K3 surfaces. (Alternatively, one could consider $\DD_{ \Lambda}^+$, one of the two connected components of $\DD_{\Lambda}$ and $\Gamma^+$, the index-$2$ subgroup of $\Gamma$ that preserves $\DD_{\Lambda}^+$. If there is an element of $\Gamma$ that interchanges $\DD_{ \Lambda}^+$ and $(-\DD_{\Lambda}^+)$, which is true in most cases, e.g. for $\operatorname{rank}\ge 5$, then $\DD_{\Lambda}/\Gamma=\DD_{\Lambda}^+/\Gamma^+$.)

The moduli functor and the moduli space in an essential way depend on the choice of a vector $h\in\Lattice_\RR$ with $h^2>0$. So they should be accurately called the moduli of $(\Lattice,h)$-(quasi)polarized K3 surfaces. The coarse moduli spaces, however, do not depend on $h$, but only on $\Lattice$.

One fixes the positive cone $\sC_\Lattice$, one half of the set of vectors in $\Lattice_\RR$ with positive square and an open cone $\sK_\Lattice\subset\sC_\Lattice$ whose closure $\overline{\sK}_\Lattice$ is a fundamental domain for the Weyl group $W(\Lattice)$ generated by the reflections in the roots of $\Lattice$, vectors $\alpha\in\Lattice$ with $\alpha^2=-2$. This positive chamber has a further locally finite subdivision into the open \emph{small cones} and their faces, the \emph{generalized small cones}. For a vector $h\in\sC_\Lattice$ with $h^2>0$ we make an auxiliary choice of a small cone $\sigma$ such that $h\in\bar\sigma$. Then one has the moduli functor $\sF^q_{\Lattice,\sigma}$ of $(\Lattice,\sigma)$-quasipolarized K3 surfaces with the coarse moduli space $F^q_{\Lattice,\sigma}$. Its points are smooth K3 surfaces $S$ together with a primitive lattice embedding 
$$j\colon \Lattice\to \Pic S$$
such that $j(h')$ is big and nef for every $h'\in\sigma$. The set of classes $x\in\Lattice_\RR$ for which $j(x)$ is nef is the closure $\bar\sigma$. The functor $\sF^q_{\Lattice,\sigma}$ is highly nonseparated but its coarse moduli space is separated and it is canonically isomorphic to the quasi-projective variety $F_\Lattice=\DD_\Lattice/\Gamma$, see \cite[Corollary 4.20]{AE25}. 

The functor $\sF_{\Lattice,h}$ of $(\Lattice,h)$-polarized K3 surfaces with the coarse moduli space $F_{\Lattice,h}$ is separated. Its points are K3 surfaces $\oS$ with $ADE$ singularities which are obtained from the $(\Lattice,\sigma)$-quasipolarized surfaces $S$ by a contraction defined by a linear system $|mh|$ for $m\gg0$. This definition does not depend on the auxiliary choice of a small cone $\sigma$, see \cite[Lemma 5.8]{AE25}. One has a canonical isomorphism $F_{\Lattice,h}=F_\Lattice=\DD_\Lattice/\Gamma$ as well, see \cite[Corollary 5.9]{AE25}. 

For a unimodular lattice, such as $\Lambda=U\oplus E_8$, there is only one small cone in a Weyl chamber, i.e.\ $\sigma = \cK_\Lambda$. Indeed, by \cite{AE25} the walls of the small cones are $\beta^\perp$ for $\beta\in L_{K3}$ with $\beta^2=-2$ such that $\langle \Lattice,\beta\rangle$ is hyperbolic and $\beta^\perp\cap\cC_\Lambda\ne\emptyset$. Writing $\beta=\beta_{\Lattice} + \beta_{\Lattice^\perp}$ with $\beta_{\Lattice}\in\Lattice$ and $\beta_{\Lattice^\perp}\in\Lattice^\perp$, one must have $\beta_{\Lattice}^2<0$ and $\beta_{\Lattice^\perp}^2\le 0$. This easily implies that $\beta_{\Lattice^\perp}=0$ and
$\beta\in\Lambda$, so the hyperplanes $\beta^\perp$ are exactly the walls of the $W(\Lattice)$-Weyl chambers.

\subsection{Lattice-polarized K3 surfaces for $\Lattice=U\oplus E_8$}\label{sec: F_Lambda}

We now fix our unimodular lattice to be $\Lattice =\II_{1,9} \simeq U\oplus E_8$, an even unimodular lattice of signature $(1, 9)$. Its orthogonal complement is $\Lattice^\perp = \II_{2,10} \simeq U^2\oplus E_8$. By the above, we have a quasi-projective variety 
$$F_\Lattice=\DD_\Lattice/\Gamma, \;\;\; 
\Gamma=O(\Lattice^\perp),
$$
which is both the  moduli space of $(\Lattice,\sigma)$-quasipolarized smooth K3 surfaces and the  moduli space of $(\Lattice,h)$-polarized K3 surfaces with $ADE$ singularities for the appropriate choices of $\sigma$ and $h$.

Let us fix a basis $\{\alpha_0,\alpha_1,\dotsc,\alpha_9\}$ of $\Lattice$ such that all $\alpha_i^2=-2$ and the dual graph of the intersection form is the $T_{2,3,7}$ diagram of Definition~\ref{defn:T237}, the $\alpha_i$ corresponding to the $\Theta_i$. 
The hyperplanes $\alpha_i^\perp$ in $\Lattice\otimes\RR$ support the facets of the positive Weyl chamber $\sK_\Lattice\subset\sC_\Lattice$ for the $W(\Lattice)$-action, and $\cK_\Lambda=\sigma$ is a small cone itself. Let us fix the following vectors
\begin{eqnarray*}
    &h = 6\alpha_0 +  12\alpha_1 + 18\alpha_2 + 24\alpha_3 + 30\alpha_4 + 36\alpha_5 + 42\alpha_6 + 28\alpha_7 + 14\alpha_8 + 21\alpha_9\\
    & s = \alpha_0\qquad
    f = \alpha_1 + 2\alpha_2 + 3\alpha_3 + 4\alpha_4 + 5\alpha_5 + 6\alpha_6 + 4\alpha_7 + 2\alpha_8 + 3\alpha_9
\end{eqnarray*}  
One has $h\cdot \alpha_6=1$, $h\cdot\alpha_i=0$ for $i\ne 6$, and $h^2=42>0$, so $h\in\sC_\Lattice\cap\bar\sigma$; it lies in a face of the small cone $\sigma=\sK_\Lattice$. One also has $f\cdot \alpha_0=1$, $f\cdot\alpha_i=0$ for $i>0$, and $f^2=0$, so that $f$ lies in $\bar\sigma$ but not in the positive cone $\sC_\Lattice$.
The Weyl chamber $\sK_\Lattice$ admits a Weyl vector $h'$ satisfying $h'\cdot \alpha_i=1$, but this fact is not essential to us.

\begin{prop}\label{prop:F-moduli-functor}
    The moduli functor $\sF^q_{\Lattice,\sigma}$ parametrizes smooth  K3 surfaces $S$ with a $T_{2,3,7}$ configuration of irreducible smooth rational curves $\Theta_0,\dotsc,\Theta_9$ and with an elliptic fibration $\pi\colon S\to\bP^1$ which has a distinguished section $\Theta_0$ and a distinguished $\II^*$-fiber     with support  $\cup_{i=1}^9\Theta_i$ ($\tE_8$ in Dynkin notation).

    The moduli functor $\sF_{\Lattice,h}$ parametrizes K3 surfaces $\oS$ with $ADE$ singularities obtained from the above smooth K3 surfaces $S$ by contracting the curves $\Theta_i$ for $i\ne 6$ and the curves in the other fibers of the elliptic fibration which are disjoint from the section $\Theta_0$.
\end{prop}
\begin{proof}
    For a surface $(S,j)\in \sF^q_{\Lattice,\sigma}$ we will identify $\Lattice$ with a subset of $\Pic S$ via~$j$. 
    By Riemann--Roch, the root $\alpha_i$ is represented by an effective or anti-effective divisor (supported on irreducible $(-2)$-curves).
        By definition, any line bundle $h'\in\sK_\Lattice$ is big and nef. Since $h'\cdot\alpha_i>0$, each $\alpha_i$ is thus represented by an effective curve, denote it by $\Theta_i$. 
    Here $\alpha_i$ is a positive root of $\Pic S$ with respect to $h'$, so it is the sum of simple roots. A simple root is represented by an irreducible $(-2)$-curve, but \emph{a priori} $\Theta_i$ may be reducible. 
    We shall now prove that each $\Theta_i$ is irreducible.

    Since $f\in\bar\sigma$, it is nef and effective. Since also $f^2=0$, the linear system $|f|$ defines an elliptic fibration $\pi\colon S\to\bP^1$ by Riemann--Roch. 
    Since $f\cdot\alpha_0=1$, $f$ is primitive, so it features as a fiber of $\pi$.
    The condition $f\cdot\alpha_i=0$ for $i>0$ implies that $\cup_{i>0}\Theta_i$ 
    supports
    a fiber of $\pi$. The dual graph of this fiber has a contraction to the graph $\tE_8$ defined by the matrix $\Theta_i\cdot\Theta_j$. Since the $\tE_8$-configuration is maximal among the Kodaira's configurations of fibers of elliptic fibrations, it follows that the curves $\Theta_1,\dotsc, \Theta_9$ are irreducible.

    Since $f\cdot \alpha_0=1$, one has $\Theta_0 = \oTheta_0+E$, $E=\sum a_{0j}E_{0j}$, where $\oTheta_0$ is a section of~$\pi$, and $E_{0j}$ are some $(-2)$-curves lying in fibers of~$\pi$.  We have
    \begin{displaymath}
        -2 = s^2 = (\oTheta_0+E)^2 = \oTheta_0^2 + 2\oTheta_0 \cdot E + E^2,
    \end{displaymath}
    which implies $2\oTheta_0 \cdot  E + E^2=0$ since $\oTheta_0^2=-2$ for a section of a K3 elliptic fibration. Now,
    \[
    (s+2f) \cdot  E = s \cdot E = (\oTheta_0+E) \cdot  E = \tfrac12 E^2.
    \]
   We shall now estimate $E^2$ in two ways in order to infer that $E=0$.
    Since 
    \[
    \Nef(S)\cap\Lattice_\bR = \bar\sigma =\oK_\Lambda= \{v \in \Lattice_\bR \mid v\cdot\alpha_i\ge0\}
    \]
    and 
    $(s+2f) \cdot \alpha_i\ge0$ for $0\le i\le 9$, the line bundle $s+2f$ is nef on $S$.
    Thus, $E^2\ge0$. Since the intersection form on any fiber of $\pi$ is negative semidefinite, 
    we have $E^2\leq 0$. Together,
    it follows that $E^2=0$ and $E$ is linearly equivalent to a multiple of fiber by Zariski's lemma. Now, $2\oTheta_0 \cdot E = -E^2=0$, which implies that $E=0$. Hence, $\Theta_0=\oTheta_0$ is an irreducible curve as well.

    This completes the proof for the moduli functor $\sF^q_{\Lattice,\sigma}$ for the smooth quasipolarized K3 surfaces $S$. By definition, the surfaces parametrized by the moduli functor $\sF_{\Lattice,h}$ are the K3 surfaces $\oS$ with $ADE$ singularities obtained from these by the linear system $|mh|$ for $m\gg0$. Clearly, this linear system contracts the curves $\Theta_i$ for $i\ne 6$ and all the other $(-2)$-curves which are disjoint from $\cup_{i=0}^9\Theta_i$. These are  the curves in the other fibers of $\pi$ which are disjoint from the section $\Theta_0$.
\end{proof}
    
\begin{cor}\label{cor:normalize M^K3}
     For any $c\in (0,1]$ there is an isomorphism $\phi\colon F_\Lattice\to M_{c,\red}^{\rm K3}$.   
\end{cor}

\begin{proof}
    By Theorem~\ref{thm: same reduced} the space $M_{c,\rm red}$ does not depend on $c$, so we may assume $c\le \frac7{13}$. Now, by Proposition~\ref{prop:F-moduli-functor}, the families of K3 surfaces $\oS$ parameterized by $F_\Lattice$ are exactly the families of stable surface pairs $(X,c\oTheta_6)$ of K3 type $\I$ in Theorem~\ref{thm: liu2017minimal-volume} for $c\le \frac7{13}$. This gives isomorphisms of moduli functors, thus an isomorphism $F_\Lattice\to M_{c , \red}^{\Kth}$ between their  moduli spaces, for $c\le \frac7{13}$.
\end{proof}

\subsection{Compactifications of $F_\Lattice$ and their maps to $M_c$}
\label{ss:comp}

In this subsection, we extend the bijective morphism $\phi\colon F_\Lattice\rightarrow M_c^\Kth$ of Corollary~\ref{cor:normalize M^K3} to compactified moduli spaces via the KSBA compactification, and prove Corollary~\ref{cor: main} up to normalization.

On the moduli space $F^q_{\Lattice,\sigma}\simeq F_\Lattice$ of $(\Lattice,\sigma)$-quasipolarized smooth K3 surfaces we have the following canonical choice of a  big and nef divisor $R\in|L|$, $L=j(h)$:
\[
    R = 6\Theta_0 +  12\Theta_1 + 18\Theta_2 + 24\Theta_3 + 30\Theta_4 + 36\Theta_5 + 42\Theta_6 + 28\Theta_7 + 14\Theta_8 + 21\Theta_9
\]
If $\varphi\colon S\to\oS$ is the contraction defined by $|mL|$ for $m\gg0$ to a K3 surface with $ADE$ singularities, then $R = \varphi^*(\oR)$ and $\oR=42\oTheta_6$ is a canonical choice of an ample divisor on $\oS$. Thus the pair $(\oS,\epsilon\oR)$ corresponds to the general-type pair $(\oS,c\oTheta_6)$ with $c=42\epsilon$. As in Section~\ref{sec:generalities_lattice_moduli}, the moduli functor $\sF_{\Lattice,h}$ of $\Lattice$-polarized K3 surfaces with $ADE$ singularities is separated and is coarsely represented by the same space $F_{\Lattice,h}\simeq F_\Lattice$.

For $\epsilon \ll1$ the pair $(\oS,\epsilon \oR)$ has klt singularities. This gives an embedding of $F_\Lattice$ into the moduli space of KSBA-stable pairs. In this $K$-trivial case, we can be more specific: it gives an embedding $F_\Lattice \to M({h^2},\epsilon)$ into the moduli space of \emph{KSBA stable $K$-trivial pairs} of type $(h^2, \epsilon)$ for some $0<\epsilon\ll 1$ (\cite[Definitions~3.6 and 3.7]{AET23}), which is proper and does not depend on $\epsilon$ by \cite[Theorem~3.13]{AET23}. 

\begin{defn}\label{def:oFR}
   Let $\oF_\Lattice^R$ be the normalization of the closure of $F_\Lattice$ in $M({h^2},\epsilon)$.
\end{defn}

In Theorem~\ref{thm:stable-limits} we find all KSBA stable limits of one-parameter families $(\osS,\epsilon\osR)$. In particular, this describes the points of $\oF_\Lattice^R$. As a corollary, we see that there is a bijective map $\phi\colon \oF_\Lattice^R\to\oF_\Lattice^\rBB$ to the Baily-Borel compactification. Since $\oF_\Lattice^R$ is normal, $\phi$ is the normalization map. We begin by describing $\oF_\Lattice^\rBB$.

\begin{lem}\label{lem:BB}
    The boundary $\partial F_\Lattice^\rBB$ of the Baily-Borel compactification is homeomorphic to the projective $j$-line $\bP^1_j$. It has one $0$-cusp and one $1$-cusp.
\end{lem}
\begin{proof}
    The $0$-cusps of $F_\Lattice^\rBB$ are in bijection with the isotropic lines $I=\ZZ e\subset \Lattice^\perp$ modulo $\Gamma=O(\Lattice^\perp)$. Any such primitive isotropic vector $e\in \Lattice^\perp$, $e^2=0$, is contained in a copy of $U$. Since $U$ is unimodular, it splits off in a direct sum $\Lattice^\perp=U\oplus U^\perp$. Here, $U^\perp$ is an even unimodular lattice of signature $(1,9)$. There is only one such lattice, so $U^\perp\simeq U\oplus E_8$. Any two decompositions $\Lattice^\perp=U\oplus U^\perp$ obviously differ by an isometry of $U^\perp$. And any two isotropic vectors in $U$ differ by an isometry of~$U$. Hence, there is a unique $O(\Lattice^\perp)$-orbit of primitive isotropic vectors in $\Lattice^\perp$.

    Similarly, the $1$-cusps of $F_\Lattice^\rBB$ are in bijection with the primitive isotropic planes $\ZZ^2\simeq J\subset\Lattice^\perp$ modulo $O(\Lattice^\perp)$. We repeat the argument above to show that $J\subset U^2$ and $\Lattice^\perp = U^2\oplus (U^2)^\perp$. The lattice $(U^2)^\perp$ is an even unimodular lattice of signature $(0,8)$, and there is a unique such lattice $E_8$. In the same way as above, we conclude that there is a unique $O(\Lattice^\perp)$-orbit of primitive isotropic planes in $\Lattice^\perp$. Finally, by a standard computation (see \cite[Proposition 5.5.3]{scattone1987on-the-compactification-of-moduli} or \cite{kondo1993on-the-kodaira-dimension})
    the direct sum splitting $\Lattice^\perp = U^2\oplus (U^2)^\perp$ implies that the $1$-cusp is isomorphic to the $j$-line $\bA^1_j$. Therefore, $\partial F_\Lattice^\rBB\simeq\bA^1_j\sqcup \{\rm pt\}$ is homeomorphic to $\bP^1_j$.
\end{proof}

\begin{thm}\label{thm:stable-limits}
    Any one-parameter degeneration $\bar f^0\colon (\osS^0, \epsilon\osR^0)\to \Delta\setminus 0$ over a smooth curve $(\Delta,0)$, possibly after a finite base change $(\Delta',0)\to (\Delta,0)$, can be completed to a family of KSBA stable pairs in which the central fiber $(\oS_0,\epsilon \oR_0)$ corresponds to one of the stable surface pairs $(X,c\oTheta_6)$ of Theorem~\ref{thm: liu2017minimal-volume} and Lemma~\ref{lem: X22 X211}, where $c=42\epsilon\le\frac7{13}$, of the following Types:
    \begin{enumerate}
        \item[(I)] $\oS_0$ is a K3 surface with $ADE$ singularities.
        \item[(II)] $\oS_0$ is a rational surface with a simple elliptic singularity.
        \item[(III)] $\oS_0$ is a rational surface with a cusp singularity.
    \end{enumerate}
    For each of these surfaces the minimal resolution of singularities $S_0$ has a $T_{2,3,7}$ configuration of curves $\Theta_0,\dotsc,\Theta_9$, and $\oS_0$ is obtained from it by contracting the curves $\Theta_i$ for $i\ne 6$ and the curves in the fibers of the elliptic fibration $\pi\colon S_0\to\bP^1$, other than the $\II^*$-fiber $\cup_{i>0}\Theta_i$, which are disjoint from the section $\Theta_0$. 

    In Type ($\III$), the pair $(\oS_0,\epsilon \oR_0)$ is unique up to an isomorphism. In Type $\II$, for each $j$-invariant of the elliptic singularity the pair $(\oS_0,\epsilon \oR_0)$ is unique.
\end{thm}
\begin{proof}
    After a finite base change the family $f^0$ admits a simultaneous resolution of singularities. In order not to complicate the notation, let us denote the base again by $\Delta\setminus 0$, and denote the family of quasipolarized smooth K3 surfaces by $f^0\colon (\sS^0,\epsilon \sR^0)\to \Delta\setminus 0$. 

    Our first goal is to construct (again, after a finite base change) a nice Kulikov model $f\colon (\sS,\epsilon\sR)\to \Delta$ such that the divisor $\sR$ is relatively nef and does not contain any strata of the central fiber $S_0$. Such models are called \emph{divisor models}, see \cite[Definition~2.3]{ABE22} or \cite[Definition~3.11]{AE23}. Once we have such a model, the KSBA stable model is obtained from it by a contraction by the linear system $|m\sR|$ for $m\gg0$. 

    The paper \cite{ABE22} of Alexeev-Brunyate-Engel almost accomplishes this goal: it constructs divisor models for all one-parameter degenerations of elliptic K3 surfaces for the divisor $R^\rc = s+m \sum f_i$, where $s$ is the distinguished section,  $f_i$ are the singular fibers of the elliptic fibration $\pi\colon S\to\bP^1$ counted with multiplicities (the sum of the multiplicities is $24$), and $m\ge1$ is any positive integer. Up to taking a multiple, our divisor $R$ is a subdivisor of $R^\rc$ since $R$ contains the distinguished section and the components $\Theta_i$ of the distinguished $\II^*$-fiber (which is counted with multiplicity $10$ in $\sum f_i$). So we start with the divisor model for $R^\rc$ provided by \cite{ABE22} and modify it to a divisor model for $R$.

\smallskip
    \emph{Type $\I$.} This is the case when the central fiber $S_0$ of a Kulikov model is a smooth K3 surface. In this case the entire family is in $F_{\Lattice,h}$, so it is covered by Corollary~\ref{cor:normalize M^K3}.

\smallskip
    \emph{Type $\III$.} In the case of a Type $\III$ degeneration, by \cite[Section 7B]{ABE22} there is a divisor model $(\sS,\epsilon \sR^\rc)$ with the following properties (the types $X_k,Y_k,I_k$ are defined in \cite[Section 7B]{ABE22}):
    \begin{enumerate}
        \item The central fiber $S_0$ with the irreducible components $V_j$, $1\le j\le n$ admits a fibration $S_0=\cup W_i\to C=\cup C_i$ to a chain of $\bP^1$'s with a distinguished section $s$, which is a Cartier divisor on $S_0$ extending to a divisor of sections on $\sS$. (Recall that $\sS$ is a Kulikov model, so it is a smooth $3$-fold.)
        \item A surface $W_i$ is not necessarily irreducible; it may be a cycle of ruled surfaces. Over each $C_i\simeq\bP^1$, the generic fiber of the fibration $W_i\to C_i$ is either an elliptic curve or a cycle of rational curves. 
        \item The terminal surfaces $W_1$ and $W_n$ are of types $X_k$ or $Y_k$; the intermediate surfaces $W_2,\dotsc,W_{n-1}$ are of types $I_k$.
        \item Some of the surfaces $W_j$ support ``very singular fibers" $f_i$, which are the limits of the singular fibers on $\sS_t$, $t\ne0$ (some of which may come together). For each of them, there is a divisor on $\sS$, let us denote it by $F_i$, such that $f_i = F_i\cap S_0$.
    \end{enumerate}

    Specializing this to the family of elliptic surfaces with an $\II^*$-fiber, this implies that the surface at one end is irreducible of the maximal type $X_{11}$ ($E_8$ in Dynkin notation). Let us denote it by $V_1=W_1$ and say that it is ``the left one". This is a rational elliptic surface with a section, a $\II^*$-fiber and an $I_1$-fiber which is the double curve $D_{12} = V_1\cap W_2$. 
    Thus,  $V_1\simeq X_{211}$ in the notation of \cite{miranda1986on-extremal-rational}, 
    see the discussion below.
    Each of the components $\Theta_i$, $i>0$ of the $\II^*$-fiber and the section $s=\Theta_0$ is a Cartier divisor on $S_0$ and extends to a divisor on the total space $\sS$. We picture the surface $S_0$ in Figure~\ref{fig:divisor-model1}. On the left, we show $S_0$ together with the divisor $R_0^\rc$; on the right we modify this divisor to $R_0$ by forgetting all fibers $f_i$ except the $\II^*$-fiber $f_1$ in $V_1$ and adjusting the coefficients of $s$ and $f_1$ so that they match our divisor $R=6\Theta_0 + 12\Theta_1 + \dotsc$. 
    
    The double curves $D_{ij}=V_i\cap V_{j}$ are shown in blue, and the divisors $R^\rc$ and $R$ in red. 
    We show the self intersections of the sections $s_i=s\cap V_i$, they are $-1,0,\dotsc,0,-1$. Recall that the dual complex of a Type $\III$ Kulikov degeneration is a triangulation of a sphere, so one should consider this picture drawn on $S^2$.

    \begin{figure}
        \centering
        \IfFileExists{fig1.png}{\includegraphics[width=27pc]{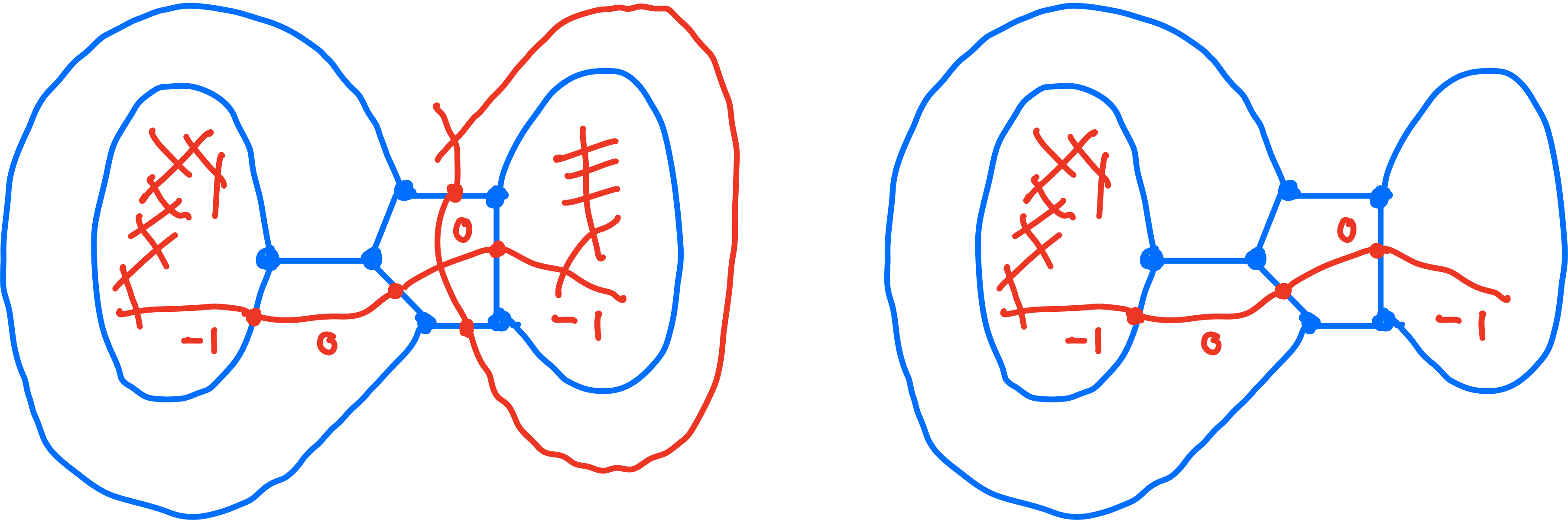}}{\fbox{\parbox[c][10em][c]{27pc}{\centering Missing figure: fig1.png}}}
        \caption{Divisor model for $(S_0,\epsilon R_0^\rc)$ and the modified pair $(S_0, \epsilon R_0)$}
        \label{fig:divisor-model1}
    \end{figure}
    \begin{figure}
        \centering
        \IfFileExists{fig2.png}{\includegraphics[width=27pc]{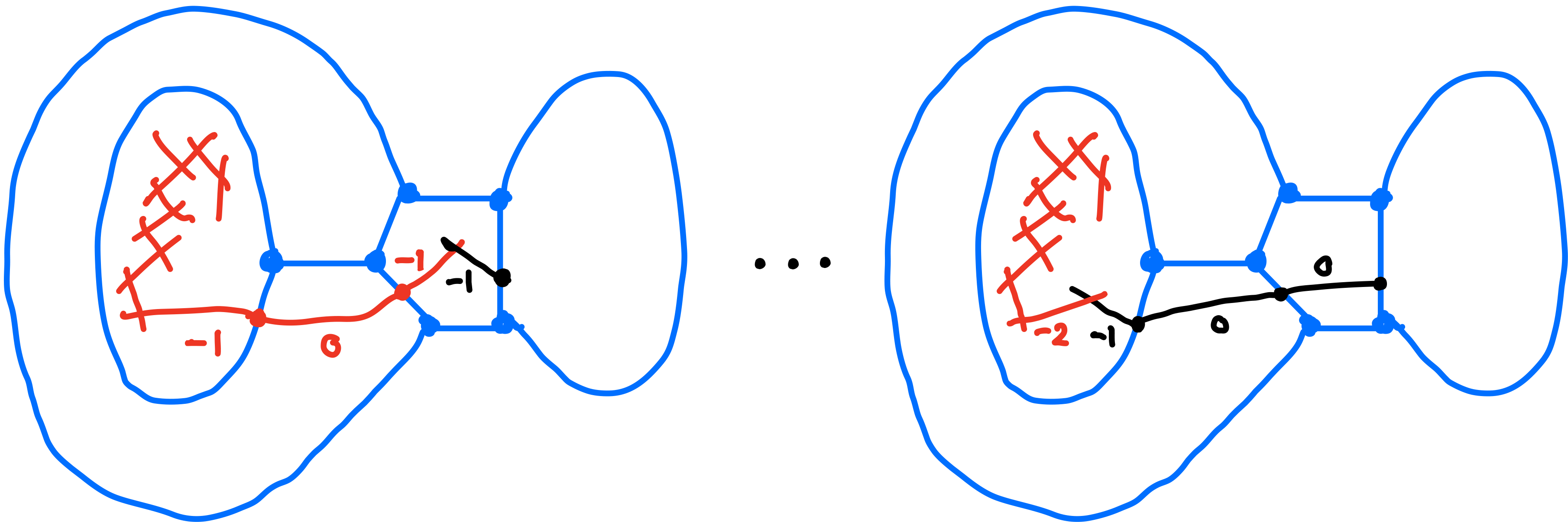}}{\fbox{\parbox[c][10em][c]{27pc}{\centering Missing figure: fig2.png}}}
        \caption{M1 modifications achieving a divisor model for $(S_0,\epsilon R_0)$}
        \label{fig:divisor-model2}
    \end{figure}  

    We observe that the divisor $R_0$ is Cartier and that its restrictions to the surfaces $W_1,\dotsc, W_{n-1}$ are nef, but the restriction to $W_n$ is not nef, since it is a $(-1)$-curve. We now modify the threefold $\sS$ by a sequence of flops which are known as M1-modifications. Recall that an M1 modification starts with two irreducible components $V_i,V_j$ intersecting along a double curve $D_{ij}=V_i\cap V_j$, then contracts a $(-1)$-curve $E_i$ on $V_i$ intersecting $D_{ij}$ at a point $p\in D_{ij}$ to this point, and finally blows up the same point $p$ on the surface $V_j$. In the new $3$-fold $\sS'$ the central fiber is the same except for the surfaces $V'_i, V'_j$, and one has $V_i = \Bl_p V'_i$, $V_j' = \Bl_p V_j$. 

    We perform this M1 modification with the surfaces $V_n\subset W_n$, $V_{n-1}\subset W_{n-1}$ on the right and the section $E:=s_n=s\cap W_n$, which is a $(-1)$-curve. On the new $3$-fold, the flat limit $s$ of the sections no longer intersects $V_n$; on $V_{n-1}$ it becomes a $(-1)$-curve $s_{n-1}$ instead of a $0$-curve; and there is a new $(-1)$-curve in the double curve $D_{n-1,n}$. We picture this flop in Figure~\ref{fig:divisor-model2}. We repeat this process moving all the way to the left until $s=s_1=\Theta_0$ is contained only in the left surface $V_1$, as in the last panel of Figure~\ref{fig:divisor-model2}. 
    Note that the curve $\Theta_0$ now becomes a $(-2)$-curve instead of a $(-1)$-curve. In particular, the curves $\Theta_0,\dotsc,\Theta_9$ on $V_1$ form the $T_{2,3,7}$-configuration.
    Also note that the $I_1$-fiber $D_{12}=V_1 \cap V_2$ becomes reducible: it equals the strict transform of $D_{12}$ plus an additional $(-1)$-curve intersecting $\Theta_0$.

    This is a divisor model for $(\sS,\epsilon \sR)$. Indeed, the divisor $R_0$ is now nef and does not contain any strata of $S_0$. To find the KSBA limit, it remains to apply the linear system $|m\sR|$ for $m\gg0$ relatively over $\Delta$. As we have already done several times in this paper, this is an application of the Cohomology and Base Change Theorem: We easily check that the higher cohomologies of $m\sR_t$ on each fiber vanish and that the linear systems $|m\sR_t|$ are base point free. This implies that there is a contraction $\sS\to\osS$ over $\Delta$. Over $\Delta\setminus 0$ we obtain the base change of the original family of K3 pairs by the finite base changes that we have done in the previous steps. As for the central fiber $\oS_0$, we get a contraction $\varphi\colon S_0\to \oS_0$. The divisor $\oR_0$ is entirely contained in $V_1$. On the surface $V_1=W_1$ the divisor $\oR_0$ is big and nef.    
    Therefore, a linear system $|m\oR_0|$ for $m\gg0$ contracts:
    \begin{itemize}
        \item the surfaces $W_2,\dotsc, W_n$ to a single point;
        \item the curves $\Theta_i$, $i\ne 6$, producing $A_1,A_2$ and $A_6$ singularities;
        \item the strict transform of $D_{12}$, producing a cusp singularity of degree $1$.
    \end{itemize}
    This is precisely the rational surface $\oS$ of Theorem~\ref{thm: liu2017minimal-volume} and Lemma~\ref{lem: X22 X211} for $c\le\frac7{13}$ with a cusp singularity.

    The minimal resolution $S$ of $\oS$ is a rational elliptic surface with a $\II^*$ fiber and a $\I_1$ fiber. By \cite{miranda1986on-extremal-rational}, there is a unique such surface denoted there by $X_{211}$, with $\II^*\I_1\I_1$ fibers. $D_{12}$ must be one of the two $\I_1$ fibers, which differ by an automorphism of $\oS_0$. 
    So there is a unique isomorphism class of the pairs $(\oS_0,\epsilon\oR_0)$. 

    \smallskip
    \emph{Type $\II$.} This case is very similar to the Type $\III$ case, with the following modifications: By \cite[Section 7F]{ABE22}, specialized to a family of elliptic surfaces with a $\II^*$-fiber, there exists a divisor model for $(\sS,\epsilon \sR^\rc)$ with the central fiber $S_0=V_1\cup\dotsb\cup V_n$ where the end surfaces $V_1$ and $V_n$ have type $X_{12}$ ($\tE_8$ in Dynkin notation),
    $E=D_{1,2}=\dotsb = D_{n-1,n}$ is a smooth elliptic curve which has some $j$-invariant $j(E)\in\bA^1_j$,
    and $V_i=\bP^1\times E$ for $2\le i\le n-1$. Each of $V_1$ and $V_n$ is a smooth rational elliptic surface; by \cite{miranda1986on-extremal-rational}, it is isomorphic to 
     \begin{itemize}
         \item 
         $\pi\colon X_{211}\to\bP^1$ with  $\II^*\I_1\I_1$ fibers and nonconstant $j$-function $\bP^1\to\bP^1$
         or
         \item 
         $\pi\colon X_{22}\to\bP^1$ with  $\II^*\II$ fibers and $j\equiv 0$. 
     \end{itemize} 

    In the first case the $j(E)\ne0$ and the limit of the singular fibers of the elliptic fibrations $\sS_t\to\bP^1$ is the $\II^*$-fiber (counted with multiplicity $10$) and the two $I_1$-fibers (each counted with multiplicity $1$). In the second case $j(E)=0$ and the limit of the singular fibers is the $\II^*$-fiber (counted with multiplicity $10$) and the $\II$-fiber (counted with multiplicity $2$).

    The rest of the construction is exactly the same as in type $\III$. After $n-1$ M1 modifications, the limit of the section $s$ is the $(-2)$-curve $\Theta_0\subset V_1$, and fiber $D_{12}$ transforms into the strict transform of $D_{12}$ plus an additional $(-1)$-curve intersecting $\Theta_0$. In particular, the curves $\Theta_0,\dotsc,\Theta_9$ on $V_1$ again form the $T_{2,3,7}$-configuration.

    The KSBA-stable model of this divisor model is obtained by a birational morphism $\varphi\colon S_0\to\oS_0$.
    The divisor $\oR_0$ is entirely contained in $V_1$. On the surface $V_1$ the divisor $\oR_0$ is big and nef. So $|mR_0|$ for $m\gg0$ contracts: 
     \begin{itemize}
        \item the surfaces $V_2,\dotsc,V_n$ to a point;
        \item the curves $\Theta_i$, $i\ne 6$, producing $A_1,A_2$ and $A_6$ singularities;
        \item the strict transform of $D_{12}$, producing a simple elliptic singularity of degree~$1$ with the $j$-invariant $j(D_{12})$.
    \end{itemize}   

     The surface $S_0$ is a rational elliptic surface with an $\II^*$ fiber, 
     say over $\infty\in\PP^1$.
     Recall the two possibilities for it discussed above. 
     
     If $S_0 = \Bl_1 V_1\cong \Bl X_{211}$, the $j$-function has degree $2$,
     but note  the two special branch points: 
     the j-invariant $j=0$ is only assumed once, namely by the singular fibre $\II^*$ at $\infty$,
     while the j-invariant $j=12^3$ is only attained by a
     single smooth fibre (the unique elliptic curve with an automorphism of order $4$).
Indeed,    
    all other values are assumed twice, and the two fibers with the same $j$-invariant differ by an automorphism of $X_{211}$ commuting with $\pi$ (which induces an order four automorphism on the two ramified fibres).    
    
   If $S_0=\Bl_1 V_1 \cong \Bl X_{22}$, the $j$-invariant of every fiber is $0$ and two such fibers differ by an automorphism of  $X_{22}$ commuting with $\pi$. So in both cases the isomorphism class of the pair $(S_0,D_{12})$ is uniquely determined by $j(D_{12})$. Thus, the isomorphism class of the pair $(\oS_0,\epsilon \oR_0)$ is uniquely determined by $j(D_{12})$. 

     This completes the proof.
\end{proof}

\begin{cor}\label{cor:oFR=oFBB}
    $\oF_\Lattice^R=\oF_\Lattice^\rBB$, the Baily--Borel compactification of $F_\Lattice$.
\end{cor}
\begin{proof}
    By Definition~\ref{def:oFR}, $\oF_\Lattice^R$ is the normalization of the closure of $F_\Lattice$ in the KSBA moduli space $M(h^2,\epsilon)$. By \cite[Theorem~3.17]{AET23} there is a regular morphism $\phi\colon \oF_\Lattice^R\to \oF_\Lattice^\rBB$. We claim that $\phi$ is a bijection. Indeed, by Theorem~\ref{thm:stable-limits}, a point in the boundary of $\oF_\Lattice^R$ corresponding to a stable limit of a family $(\osS,\epsilon\osR)$ of K3 pairs is uniquely determined by a point $j\in\PP^1_j$, the $j$-invariant of the curve $D_{12}$ which is contracted to the elliptic, resp. cusp singularity (we set $j=\infty$ if $D_{12}$ is a rational nodal curve). 

    By Lemma~\ref{lem:BB}, a point in the boundary of $\oF_\Lattice^\rBB$ is also uniquely determined by a point $j\in\PP^1_j$. And it is the same $j$-invariant: for the Baily-Borel decomposition, a point in the $1$-cusp corresponds to the $j$-invariant $j(E)$, where $E$ is the elliptic curve appearing in any Type $\II$ Kulikov model of a one-parameter degeneration; and one has $j=\infty$ for the $0$-cusp appearing in any Type $\III$ degeneration. By the proof of Theorem~\ref{thm:stable-limits}, $E=D_{12}$ is the same elliptic, resp. rational nodal curve.

    Since $\oF_\Lattice^\rBB$ is normal, by the Main Zariski Theorem $\phi$ is an isomorphism.  
\end{proof}

We can now prove a weaker version of Corollary~\ref{cor: main}.

\begin{proof}[Proof of Corollary~\ref{cor: main} up to normalization]
    We assume that $c\le\frac7{13}$ and set $\epsilon=\frac{c}{42}$. Then a family of stable $K$-trivial pairs $(\oS,\epsilon \oR)$ in $\oF_\Lattice^R$ is also a family of stable pairs of general type $(X, c\oTheta_6)$ of Theorem~\ref{thm: liu2017minimal-volume}.
    So we have a morphism of moduli functors and a morphism of the  moduli spaces, $\mu\colon \oF_\Lattice^R\to M_c$, and it is a bijection. By Corollary~\ref{cor:oFR=oFBB}, $\oF_\Lattice^R=\oF_\Lattice^\rBB$ is normal. So by the Main Zariski Theorem $\mu$ is the normalization map.
\end{proof}

\subsection{The recognizability of $R$}

Recall that a canonical choice of a polarizing divisor $R$ is recognizable if for one-parameter Kulikov degenerations $f\colon \sS \to (\Delta,0)$ in the quasipolarized moduli the flat limits $R_0$ of the polarizing divisors $\sR_t$ on $\sS_t$ depend only on the central fiber $S_0$.

\begin{prop}
    For the moduli space $F_\Lambda$, the divisor $R$ is recognizable.
\end{prop}
\begin{proof}
 
By \cite[Theorem~2]{AE23} the \emph{rational curve divisor} $R^\rc$ on the moduli space $F_{2d}$ of polarized K3 surfaces is recognizable, and 
by \cite{ABE22} (cf. also \cite[Example 9.21]{AE23})
on the locus of elliptic K3 surfaces with a section (which are $U$-polarized) it reduces to the divisor $R^\rc$ which is the sum of the section and the singular fibers with some coefficients.
Via the inclusion $U\to\Lattice=U\oplus E_8$ to the first summand, it gives the divisor $R^\rc$ which we used in the proof of Theorem~\ref{thm:stable-limits}. 
Clearly, the restriction of a recognizable divisor for the $U$-polarized K3 surfaces to the locus of $\Lattice$-polarized surfaces is again recognizable. Equally obvious is the fact that any subdivisor of a recognizable divisor is itself recognizable. Our divisor $R$ is a subdivisor of a multiple of $R^\rc$, so it is recognizable.
  
\end{proof}

By \cite[Theorem~1]{AE23} it follows that the compactification $\oF_\Lattice^R$ is semitoroidal, i.e.\ it dominates the Baily-Borel compactification and is dominated by some toroidal compactification. But the equality $\oF_\Lattice^R=\oF_\Lattice^\rBB$ of Corollary~\ref{cor:oFR=oFBB} already includes this.

\section{The moduli stack \texorpdfstring{$M_c$}{} for \texorpdfstring{$\frac{6}{11}\leq c\leq 1$}{}}\label{sec: M_c with c large}
By Theorem~\ref{thm: liu2017minimal-volume}, we have $M_c = M_1$ for $\frac{6}{11}\leq c\leq 1$. Thus, we may assume that $c=1$ in this section.
The treatment of $M_1$ is parallel to the case with $c$ very small, as is carried out in Section~\ref{sec: M_c with c small}. It turns out, however, that $M_{1}\cong M_\epsilon$ as stacks for $\epsilon \leq \frac{7}{13}$.

As in Section~\ref{sec: M_c with c small}, we start by computing the log canonical ring of the stable surfaces in $M_1$.
\begin{prop}\label{prop: canonical ring c large}
Let $X$ be a stable surface with $p_g(X)=1$ and $K_X^2=\frac{1}{143}$. 
    For $n\geq 0$, let $P_n:=h^0(X, nK_X)$ be the $n$-th plurigenus and $\varphi_n\colon X\dashrightarrow\PP^{P_n-1}$ the rational map induced by the linear system $|nK_X|$. Then the following holds.
    \begin{enumerate}[leftmargin=*]
     \item $\varphi_{26}$ is generically finite of degree 2 onto the image, and $\varphi_{39}$ is birational onto the image.
       \item The generating series of the plurigenera satisfies
        \[
        \sum_{n\geq 0} P_nt^n=\frac{1-t^{78}}{(1-t)(1-t^{11})(1-t^{26})(1-t^{39})}.
        \]
        \item There is an isomorphism of $\CC$-algebras
    \[
    R(X, K_X) = \bigoplus_{m\geq 0} H^0(X, m K_X)\cong \CC[w,x,y,z]/(f)
    \]
    where $(\deg w,\,\deg x,\, \deg y,\, \deg z)=(1, 11, 26, 39)$, and 
    \begin{equation}\label{eq: f degree 78}
     f=f_g(w,x,y,z):=z^2+ y^3+ x^7w + g(w,x,y)   
    \end{equation} 
    with $g$ a nonzero weighted homogeneous polynomial in $w,x,y$ coming from the subspace $V_{78}\subset \CC[w,x,y]_{78}$ spanned by the monomials in $w,x,y$ of degree $78$, not divisible by $y^2$ or $x^6$:
    \begin{equation*}\label{eq: monomial 78}
w^8x^4y,\, w^{19}x^3y,\, w^{30}x^2y,\, w^{41}xy,\, w^{52}y,\, w^{23}x^{5},\, w^{34} x^4,\, w^{45}x^3,\, w^{56}x^2,\, w^{67}x,\, w^{78}.
\end{equation*}
\end{enumerate}
\end{prop}
\begin{proof}
(i) Let $\pi\colon \tX\rightarrow X$ be the minimal resolution, as in Theorem~\ref{thm: liu2017minimal-volume}. Then, for $n\geq 0$,
\[
|nK_X|\cong |\lfloor n\pi^*K_X \rfloor| = |\lfloor n(K_{\tX}+B_{\tX}) \rfloor| 
\]
and, using the notation of Theorem~\ref{thm: liu2017minimal-volume}, there is a $\QQ$-linear equivalence:
\[
K_{\tX}+B_{\tX} \sim_\QQ E + \sum_{i=0}^5\frac{i+1}{13}\widetilde{\Theta}_i + \frac{6}{11}\widetilde{\Theta}_6+\frac{4}{11}\widetilde{\Theta}_7+\frac{2}{11}\widetilde{\Theta}_8+\frac{3}{11}\widetilde{\Theta}_9
\]
Observe that there is a $\tD_{11}\in |\lfloor 11(K_{\tX}+B_{\tX})\rfloor|$ containing the blown-up $\II^*$-fiber of the genus 1 fibration $\tilde h\colon  \tX\rightarrow \PP^1$ as a sub-divisor:
\[
F_{\tX} = 11E + \widetilde{\Theta}_1 +2\widetilde{\Theta}_2 +3\widetilde{\Theta}_3 +4\widetilde{\Theta}_4 +5\widetilde{\Theta}_5 +6\widetilde{\Theta}_6 + 4\widetilde{\Theta}_7+2\widetilde{\Theta}_8+3\widetilde{\Theta}_9.
\]
Similarly, there is a $\tD_{26}\in |\lfloor 26(K_{\tX}+B_{\tX})\rfloor|$ and $\tD_{39}\in |\lfloor 39(K_{\tX}+B_{\tX})\rfloor|$ such that 
\[
\tD_{26}\geq 2(\widetilde\Theta_0+F_{\tX}),\quad \tD_{39}\geq 3(\widetilde\Theta_0+F_{\tX}).
\]

For $m\in \{2,3\}$, consider the long exact sequence of cohomology associated to the short exact sequence 
 \[
 0\rightarrow \sO_{\tX}((m-1)\widetilde{\Theta}_0+ mF_{\tX})\rightarrow \sO_{\tX}(m\widetilde{\Theta}_0+mF_{\tX})\rightarrow \sO_{F_{\tX}}(m\widetilde{\Theta}_0+mF_{\tX})\rightarrow 0.
 \]
The Leray spectral sequence gives 
\begin{multline*}
H^1(\tX, \sO_{\tX}((m-1)\widetilde{\Theta}_0+ mF_{\tX})) =H^1(\PP^1,\tilde h_*\sO_{\tX}((m-1)\widetilde{\Theta}_0+ mF_{\tX}))\\
=
\begin{cases}
 H^1(\PP^1,\sO_{\PP^1}(2))=0 & \text{if $m=2$} \\
 H^1(\PP^1,\sO_{\PP^1}(-1))=0 & \text{if $m=3$}
\end{cases}  
\end{multline*}
where the second equality follows from \cite[II.4.3]{miranda1989the-basic-theory}. Thus we have a surjection
\[
H^0(\tX, m\widetilde{\Theta}_0+mF_{\tX})\rightarrow H^0(F_{\tX},\sO_{F_{\tX}}(m\widetilde{\Theta}_0+mF_{\tX}))\rightarrow 0.
\]
It follows that the restriction of $|m\widetilde{\Theta}_0+mF_{\tX}|$ to $F_{\tX}$ is a complete linear system of degree $m$, and hence induces an embedding (resp.~$2$-to-$1$ map) of $F_{\tX}$ if $m=3$ (resp.~$m=2$). On the other hand, since $2\widetilde{\Theta}_0+2F_{\tX}\geq 2F_{\tX}$ and $|2F_{\tX}|$ separates fibers of $\tilde h$, the linear system $|2\widetilde{\Theta}_0+2F_{\tX}|$ also separates fibers of $\tilde h$. In conclusion, $|m\widetilde{\Theta}_0+mF_{\tX}|$ induces a birational (resp.~degree $2$) map from $\tX$ if $m=3$ (resp.~$m=2$). 

Correspondingly, the pluricanonical map $\varphi_{39}$ of $X$ is birational onto the image,
and $\varphi_{26}$
is generically finite of degree $2$ onto the image.

\medskip

(ii)  By Theorem~\ref{thm: liu2017minimal-volume} and Lemma~\ref{lem: non nonnormal}, the stable surface $X$ has exactly two non-Gorenstein singularities, which are quotient singularities of type $\frac{1}{13}(1,6)$ and $\frac{1}{11}(1,7)$ respectively. We compute directly the correction terms $\delta_x(nK_X)$ to the Riemann--Roch formula for $P_n:=h^0(nK_X)$ of these two singularities for $1\leq n\leq 7$:
\begin{center}
\begin{tabular}{|c||c|c|}
\hline
&$\frac{1}{13}(1,6)$&  $\frac{1}{11}(1,7)$\\ \hline\hline
$\delta_x(2K_X)$ & $-\frac{6}{13}$ & $-\frac{6}{11}$ \\ \hline
$\delta_x(3K_X)$ & $-\frac{5}{13}$ & $-\frac{7}{11}$ \\ \hline
$\delta_x(4K_X)$ & $-\frac{10}{13}$ & $-\frac{3}{11}$ \\ \hline
$\delta_x(5K_X)$ & $-\frac{8}{13}$ & $-\frac{5}{11}$ \\ \hline
$\delta_x(6K_X)$ & $-\frac{12}{13}$ & $-\frac{2}{11}$ \\ \hline
$\delta_x(7K_X)$ & $-\frac{9}{13}$ & $-\frac{5}{11}$ \\ \hline
\end{tabular}
\end{center}
The values $\delta_x(nK_X)$ for all $n\geq 8$ can be deduced from the above data as in Remark~\ref{rmk: delta plurigenera}.
Plugging these into the Blache's Riemann--Roch formula (Theorem~\ref{thm: Blache}), we obtain $P_n$ for $1\leq n\leq 82$, as follows:
\begin{center}
\begin{tabular}{|c||c|c|c|c|c|c|c|c|}
\hline
$n$& $1$--$10$ & $11$--$21$ & $22$--$25$ & $26$--$32$ & $33$--$36$ & $37$--$38$ & $39$--$43$ & $44$--$47$   \\ \hline
$P_n$ & $1$ & $2$ & $3$ & $4$ & $5$ & $6$ & $7$ & $8$ \\ \hline\hline
$n$&  $48$--$49$ & $50$--$51$ & $52$--$54$ & $55$--$58$ & $59$--$60$ & $61$--$62$ & $63$--$64$ & $65$ \\ \hline
$P_n$ & $9$ & $10$ & $11$ & $12$ & $13$ & $14$ & $15$ & $16$ \\ \hline\hline
$n$& $66$--$69$ & $70$--$71$ & $72$--$73$ & $74$--$75$ & $76$ & $77$ & $78$--$80$ & $81$--$82$\\ \hline
$P_n$ & $17$ & $18$ & $19$ & $20$ & $21$ & $22$ & $23$ & $24$\\ \hline
\end{tabular}
\end{center}
Set 
\[
A(t)=\sum_{n\geq 0}a_nt^n:=(\sum_{n\geq 0} P_n t^n)(1-t)(1-t^{11})(1-t^{26})(1-t^{39})
\]
To verify (ii), we need to prove that $A(t)=1-t^{78}$. Using the values of $P_n$ for $n\leq 78$ computed above, one may check directly that  
\[
A(t)= 1- t^{78} + \text{higher order terms.}
\]
For $n>78$, the coefficient of $t^n$ in $A(t)=\sum_{n\geq 0}a_nt^n$ is
\begin{multline*}
 a_n=(P_n-P_{n-1}) - (P_{n-11}-P_{n-12}) - (P_{n-26}-P_{n-27}) + (P_{n-37}-P_{n-38})  \\ - (P_{n-39}-P_{n-40}) 
 + (P_{n-50}-P_{n-51}) + (P_{n-65}-P_{n-66}) - (P_{n-76}-P_{n-77})   
\end{multline*}
Observe that the correction terms from the singularities cancel each other in the above expression of $a_n$, due to their periodicity. 
Indeed, modulo $13$, this applies to the terms 
appearing underneath each other in the above formula,
and modulo $11$ we can pair $P_n$ against $P_{n-11}$ etc.

To compute the above sum, we may thus simply compare the usual Riemann-Roch expressions. 
Since the constant term $\chi(X,\mathcal O_X)$ cancels out anyway,
this amounts to studying the terms $\frac 12 (D-K_X)\cdot D$.
At $D=nK_X$ and $D=(n-1)K_X$, the difference is 
exactly $(n-1)K_X^2$.
Thus we get, a priori up to the factor of $K_X^2$,
\[
a_n = (n-1)-(n-12)-(n-27)+(n-38)-(n-40)+(n-51)+(n-66)-(n-77) = 0.
\]
To conclude, $a_n=0$ for $n> 78$, and hence $A(t)=1-t^{78}$, as required.

\medskip

(iii) Now we are ready to figure out the graded $\CC$-algebra structure of $R(X, K_X)$. By inspection of the dimensions given in (ii), we need generators $\bar w,\,\bar x,\,\bar y,\,\bar z$  in the degrees $1,11,26, 39$ respectively, and there is a relation in degree $78$ of the form
\[
f(\bar w,\bar x,\bar y,\bar z)= a\bar z^2 + \bar z\, g_{39}(\bar w,\bar x,\bar y) +  g_{78}(\bar w,\bar x,\bar y)
\]
where $a\in\CC$ is a constant and $g_{39}$ (resp.~$g_{78}$) is a weighted homogeneous polynomial of degree $39$ (resp.~$78$) in $\bar w,\bar x,\bar y$.

Since $\deg\varphi_{26}=2>1=\deg \varphi_{39}$ by (ii), $f$ is irreducible and the coefficient $a$ is nonzero. By rescaling and completing the square with respect to $\bar z$, we may assume that $a=1,g_{39}=0$, and 
\[
f(\bar w,\bar x,\bar y,\bar z)= \bar z^2 +  g_{78}(\bar w,\bar x,\bar y).
\]
Consider the following homomorphism of $\CC$-algebras from the polynomial ring
\[
\Lambda\colon \CC[w,x,y,z]\rightarrow R(X, K_X),\quad w\mapsto \bar w,\, x\mapsto \bar x, \,y\mapsto \bar y,\, z\mapsto \bar z.
\]
Then $f(w,x,y,z)\in \ker\Lambda$. 
Since $\phi_{39}$ is birational, so is the map 
$$\psi_{39}=(\bar w:\bar x:\bar y:\bar z)\colon \;\; X\dashrightarrow \PP(1,11,26,39).
$$
It follows that $\ker\Lambda$, the defining homogeneous ideal of $\psi_{39}(X)$, is generated by $f$, and $\Lambda$ induces an injection
\[
\lambda\colon \CC[w,x,y,z]/(f)\rightarrow R(X,K_X).
\]
Denote $Q_n:=\dim_{\CC} \left(\CC[w,x,y,z]/(f)\right)_n$ for $n\geq 0$. Then, by (i) and \cite[Section~18]{fletcher2000working-with-weighted}), one has
\[
 \sum_{n\geq 0}Q_n t^n = \frac{1-t^{78}}{(1-t)(1-t^{11})(1-t^{26})(1-t^{39})} = \sum_{n\geq 0}P_n t^n.
\]
It follows that $P_n=Q_n$ for any $n\geq 0$, and hence the injective homomorphism $\lambda$ is indeed an isomorphism.

Finally, we try to obtain more precise information about the weighted homogeneous polynomial $f$.
\begin{claim}\label{claim: y3 x7w appear}
The coefficients of the terms $y^3$ and $x^7w$ in $f$ must be nonzero. 
\end{claim}
In fact, if $y^3$ does not appear in $f$, then $P_y:=(0:0:1:0)\in X$, and it is the quotient of the singularity
\[
(0,0,0)\in Y_y=(f(w,x,1,z)=0)\subset \CC_{w,x,z}^3
\]
under the action of $\mu_{26}\colon (w,x,z)\mapsto (\xi_{26} w, \xi_{26}^{11}x, \xi_{26}^{39}z)$ for a primitive $26$-th root $\xi_{26}$ of $1$. But $f(w,x,1,z) = z^2 + ax^2w^4 +\text{terms of order $\geq 7$}$ for some $a\in \CC$, and one sees easily that the surface singularity $(0,0,0)\in Y_y$ is not log canonical. Now the action of $\mu_{26}$ only fixes $(0,0,0)$, we have $K_{Y_y}=\mu^*K_X$, where $\mu\colon Y_y\rightarrow X\cap (y\neq 0)$ is the quotient map. We infer that $X$ is not log canonical at $P_y$ by \cite[Proposition~5.20]{kollar1998birational-geometry}, a contradiction. 

Thus, up to rescaling the variable $y$ and completing the cube with respect to $y$, we may write $f=z^2+y^3+ h_{52}(w,x)y + h_{78}(w,x)$, where $h_{52}, h_{78}\in \CC[w,x]$, with $(\deg w,\deg x)=(1,11)$, are weighted homogeneous polynomials in $w,x$ of degrees $52$ and $78$ respectively. We have $h_{52} = b x^4w^8+\text{higher order terms}$ and $h_{78} =cx^7w+dx^6w^{12}+\text{higher order terms}$.

Assume on the contrary that $x^7w$ does not appear in $f$. Then $P_x:=(0:1:0:0)\in X$, and it is the quotient of the singularity 
\[
(0,0,0)\in Y_x=(f(w,1,y,z)=0)\subset \CC_{w,y,z}^3
\]
under the action of $\mu_{11}\colon (w,y,z)\mapsto (\xi_{11} w, \xi_{11}^{26}y, \xi_{11}^{39}z)$ for a primitive $11$-th root $\xi_{11}$ of $1$. Since $x^7w$ does not appear in $f$, in view of the expressions of $h_{52}$ and $h_{78}$ in the previous paragraph, we have $f(w,1,y,z)=z^2+ y^3 + byw^8+\text{terms of order $\geq 10$ in $w,y$}$, and one sees easily that $Y_x$ is not log canonical at the origin (compare the singularity analysis in the proof of Proposition~\ref{prop: singularity analysis c small}).  We infer as before that $P_x\in X$ is not log canonical either, a contradiction.

Now, up to rescaling the variable $x$ or $w$, we may write $f=z^2+y^3+x^7w+g_{78}(w,x,y)$, as stated in (iii). Note that $g_{78}\neq 0$, since the point $P_w=(1:0:0:0)\in (z^2+y^3+x^7w=0)$ is not log canonical. 
\end{proof}

We make now an analysis of singularities appearing in a hypersurface $X_g:=(f_g=0)\subset \PP(1,11,26,39)$ for any $g\in V_{78}\setminus\{0\}$. 

\begin{prop}\label{prop: singularity analysis c large}
Let $\CC[w,x,y]$ be the polynomial ring with $(\deg w, \deg x, \deg y)=(1,11,26)$, and $V_{78}\subset \CC[w,x,y]_{78}$ the subspace as in Proposition~\ref{prop: canonical ring c large} (iii). Then for any $g,g'\in V_{78}\setminus \{0\}$, the following holds.
\begin{enumerate}
    \item Along the divisor $D_g:=X_g\cap (w=0)$ at infinity, the surface $X_g$ has two singularities:
    \[
    \begin{cases}
     \text{a $\frac{1}{11}(1,7)$-singularity at $(0:1:0:0)$, and} \\
     \text{a $\frac{1}{13}(1,6)$-singularity at $(0:0:-1:1)$,}
    \end{cases}
    \]
    so that the inverse image of $D_g$ on the minimal resolution $\pi\colon \tX_g\rightarrow X_g$ is a blown-up $T_{2,3,7}$-configuration of smooth rational curves with dual graph described as in Theorem~\ref{thm: liu2017minimal-volume} (v).
    \item Define an isomorphism of vector spaces 
\[
\varphi\colon V_{78}\rightarrow V_{42},\quad g \mapsto w^{42}g(1, x/w^{6}, y/w^{14}),
\]
where $V_{42}$ is as in Proposition~\ref{prop: canonical ring c small} (iii).
Then for $g\in V_{78}\setminus\{0\}$, the affine part $X_g\cap (w\neq 0)\subset\PP(1,11,26,39)$ is isomorphic to the affine part $X\cap (w\neq 0)$ of the surface $X=(z^2+y^3+x^7+\varphi(g)=0)\subset \PP(1,6,14,21)$. 
    \item   $X_g$ has only log canonical hypersurface singularities, $K_{X_g}\cong \sO_{X_g}(1)$, and $X_g$ is a stable surface with volume $\frac{1}{143}$.
    \item $X_g\cong X_{g'}$ if and only if there is an $\lambda\in \CC^*$ such that
    \[
    g(w,x,y)=(\lambda.g')(w,x,y):=g'(\lambda^{-7} w,\lambda x,y).
    \]
\end{enumerate}
\end{prop}
\begin{proof}
(i) We have $D_g=(z^2+y^3=0)\subset \PP(11,26,39)$, and hence $y$ and $z$ are simultaneously zero or nonzero.  If $yz\neq 0$, then, using the Jacobian criterion, one checks that the affine hypersurfaces $Y_y:=(f(w,x,-1,z)=0)\subset\CC^3_{w,x,z}$ is smooth along $Y_y\cap (w=0)$. Note that $(yz\neq 0)\cap X_g$ is the quotient of $Y_y$ by the action 
\[
\sigma_y(w,x,z):= (\xi_{26}w,\xi_{26}^6x, \xi_{26}^{39}z) =  (\xi_{26}w,\xi_{13}^3x, -z),
\]
where $\xi_{26}$ is a primitive $26$th root of unity and $\xi_{13}=\xi_{26}^2$, and there is a unique orbit with nontrivial stabilizer, namely that of $P=(0,0,1)$. The stabilizer at $P$ is $\langle \sigma_y^2\rangle\cong \ZZ/13\ZZ$. We may take $w,x$ as the local coordinates around $P\in Y_y$, and the action of $\sigma_y^2$ around $P$ is given by $\sigma_y^2(x,w) = (\xi_{13}w, \xi_{13}^6 x)$. It follows that the image $\oP=(0:0:-1:1)\in X_g$ of $P$ is a quotient singularity of type $\frac{1}{13}(1,6)$. 

Similarly, the affine surface $Y_x:=(f(w,1,y,z)=0)\subset \CC^3_{w,y,z}$ is smooth at the point $Q=(0,0,0)$, with local coordinates $y,z$. The surface $X_g\cap (x\neq 0)$ is the quotient of $Y_x$ under the action 
\[
\sigma_x(w,y,z) := (\xi_{11}w,\xi_{11}^{26}y, \xi_{11}^{39}z)=(\xi_{11}w,\xi_{11}^{4}y, \xi_{11}^{6}z)
\]
where $\xi_{11}$ is a primitive $11$th root of unity. The image $\oQ=(0:1:0:0)\in X_g$ of $Q\in Y_x$ is a quotient singularity of type $\frac{1}{11}(4,6)$, which is the same as $\frac{1}{11}(1, 7)$.

\medskip

(ii) We have 
\begin{equation}\label{eq: affine g vs varphi(g)}
f_g(1,x,y,z)= z^2+y^3+x^7 + g(1,x,y) = z^2+y^3+x^7 + (\varphi(g))(1,x,y).
\end{equation}
The left hand side of \eqref{eq: affine g vs varphi(g)} defines $X_g\cap (w\neq 0)$ and the right hand side of \eqref{eq: affine g vs varphi(g)} defines $X\cap (w\neq 0)$, and hence they are isomorphic.

\medskip

(iii) The surface $X_g$ has only log canonical hypersurface singularities along $D_g$ by (i), and on $X_g\cap (w\neq 0)$ by (ii) and Proposition~\ref{prop: singularity analysis c small}. By adjunction, we have $K_{X_g}\sim \sO_{X_g}(1)$, which is ample. It follows that $X_g$ is a stable surface with volume 
\[
K_{X_g}^2 =\frac{78}{1\cdot11\cdot 26\cdot 39}= \frac{1}{143}.
\]

\medskip

(iv) By (iii) the stable surface $X_g$ achieves the minimal volume $\frac{1}{143}$ as in Theorem~\ref{thm: liu2017minimal-volume} (1). The computation of Proposition~\ref{prop: canonical ring c large} shows that the inclusion $X_g \subset \PP(1,11,26,39)$ is via the identification $R(X_g, K_{X_g})\cong \CC[w,x,y,z]/(f_g)$. The same holds for $X_{g'}$. Thus, $X_g\cong X_{g'}$ if and only if there is a coordinate change $\sigma$ of $\PP(1,11,26,39)$ such that $f_g(w,x,y,z) = f_{g'}\circ \sigma(w,x,y,z)$. Since $f_g$ and $f_{g'}$ share the same principal part $z^2+y^3+wx^7$ and do not contain other monomials divisible by one of $z$, $y^2$ and $x^6$, the only possibility is $\sigma(w,x,y,z) = (\lambda_1 w, \lambda_2 x,\lambda_3 y,\lambda_4 z)$ where $\lambda_i\in \CC^*$ for $1\leq i\leq 4$, and $\lambda_1\lambda_2^7 = \lambda_3^3=\lambda_4^2=1$. We may find a $78$-th root $\alpha$ of $1$ such that $\alpha^{26}=\lambda_3$ and $\alpha^{39}=\lambda_4$. Since $(\deg w, \deg x, \deg y, \deg z)=(1,11,26, 39)$ and $f_{g'}$ is homogeneous of degree $78$, 
\begin{multline*}
f_{g'}(\lambda_1w,\lambda_2x,\lambda_3 y, \lambda_4z) =f_{g'}(\lambda_1 \alpha^{-1}w,\lambda_2\alpha^{-11}x,\lambda_3\alpha^{-26} y, \lambda_4\alpha^{-39}z) \\
= f_{g'}(\lambda_1 \alpha^{-1}w,\lambda_2\alpha^{-11}x,y, z).
\end{multline*}
We have 
\[
(\lambda_1\alpha^{-1})\cdot (\lambda_2\alpha^{-11})^7 = \lambda_1\lambda_2^7=1
\]
Setting $\lambda = \lambda_2\alpha^{-11}$, we have $\lambda_1 \alpha^{-1} = \lambda^{-7}$,
and hence, as required, 
\[
f_g(w,x,y,z)=f_{g'}(\lambda_1w,\lambda_2x,\lambda_3 y, \lambda_4z)=f_{g'}(\lambda^{-7}w,\lambda x,y, z)
\]
\end{proof}

\begin{rmk}
By Proposition~\ref{prop: singularity analysis c large} (iii), $X_g=(z^2+y^3+x^7w+g=0)\subset \PP(1,11,26,39)$ and $X=(z^2+y^3+x^7+\varphi(g)=0)\subset \PP(1,6,14,21)$ can be viewed as different compactifications of the same affine surface. This is also implicit in Theorem~\ref{thm: liu2017minimal-volume}.
\end{rmk}

Now we show that $M_1$ for $c\geq \frac{6}{11}$ has the same description as for $0<c<\frac{7}{13}$; see Theorem~\ref{thm:PtoMc}.
\begin{thm}\label{thm:PtoM1}
For $\frac{6}{11}\leq c\leq 1$, there is an isomorphism of Deligne--Mumford stacks
\[
\bP(4, 10, 12, 16, 18, 22, 24, 28, 30, 36, 42)\to M_c
\]
and of their coarse moduli spaces.
\end{thm}
\begin{proof}
By Theorem~\ref{thm: liu2017minimal-volume}, $M_c=M_1$ for any $c\geq \frac{6}{11}$. By Proposition~\ref{prop: canonical ring c large} (iii) and its proof, for $[X]\in M_1$, the pluricanonical ring $R(X, K_X)$ embeds $X$ into $\PP(1,11,26,39)$ as a weighted hypersurface of degree $78$, defined by
\begin{equation}
  f(w,x,y,z)= z^2 +y^3+x^7w+ g  
\end{equation}
where $g\in V_{78}^*:=V_{78}\setminus\{0\}$ as in Proposition~\ref{prop: singularity analysis c large}.

Now we may define a family of stable surfaces $\widetilde{\sX}\subset \PP(1,11,26,39)\times V_{78}^*$ over $V_{78}^*$, such that over each $g\in V_{78}^*$, the fiber over $g$ is
\[
\widetilde{\sX}_g=(z^2+y^3+x^7w+g(w,x,y)=0)\subset \PP(1,11,26,39).
\]
We have $\widetilde{\sX}_g\cong \widetilde{\sX}_{g'}$ if and only if there is a $\lambda\in \CC^*$ such that $\lambda. g= g'$, where $\CC^*$ acts on $V$ diagonally with weights for the monomial basis of Proposition~\ref{prop: canonical ring c large} (iii) as follows:
 \begin{equation}\label{eq: weights 1}
 52,\, 130,\,156,\,208,\,234,\,286,\,312,\,362,\,390,\,468,\,546
 \end{equation}
This action can be generated by the action of $\CC^*$ on the vector space spanned by two variables $x,\, w$: $\lambda.(w,x)=(\lambda^{-7}w, \lambda x)$. Note that the kernel of this $\CC^*$-action on the family $\sX$ is $\mu_{13}$, so we obtain an effective action of $\CC^*/\mu_{13}\cong \CC^*$, with weights obtained from \eqref{eq: weights 1} by dividing $13$:
\[
 4,\, 10,\, 12,\,16,\,18,\,22,\,24,\,28,\,30,\,36,\,42.
\]
In this way, we obtain the desired moduli map 
\begin{equation}\label{eq: M1 to Mc}
\bP(4, 10, 12, 16, 18, 22, 24, 28, 30, 36, 42)\to M_1,
\end{equation}
which is bijective on the sets of closed points. One can then proceed as in the proof of Theorem~\ref{thm:PtoMc} to show that the above map is indeed an isomorphism of stacks. 
\end{proof}

\begin{rmk}
We can see the bijection of  sets $M_{c,\red}(\CC)\rightarrow M_{1,\red}(\CC)$  as follows. From a surface $X=(f_{78}(w,x,y,z)=0)\subset \PP(1,11,26,39)$, where $f_{78}=z^2+y^3+wx^7+g$ and $g\in V_{78}^*$, we obtain 
\[
f'_{42}(w,x,y,z) := w^{42}f_{78}(1,x/w^{6}, y/w^{14}, z/w^{21})\in \CC[w,x,y,z]
\]
which defines a stable surface pair $(X', (w=0))\subset \PP(1,6,14,21)$ from $M_c$. Conversely, we can convert a quasihomogeneous $f'_{42} = z^2+y^3+x^7+g\in \CC[w,x,y,z]$ with $(\deg w, \deg x, \deg y, \deg z)=(1,6,14,21)$ and $g\in V_{42}^*$ as in Proposition~\ref{prop: canonical ring c small} (iii) to a quasihomogeneous polynomial $f_{78}$ of degree $78$:
\[
f_{78}(w,x,y,z) := w^{78}f'_{42}(1, x/w^{11}, y/w^{26}, z/w^{39})\in \CC[w,x,y,z] 
\]
with $(\deg w, \deg x, \deg y, \deg z)=(1,11,26,39)$, and $X=(f_{78}=0)$ is a stable surface in $M_1$.
\end{rmk}

\section{Hyperbolicity of the bases of equisingular families in \texorpdfstring{$M_c$}{}}\label{sec: hyperbolicity}
According to \cite[Definition 1.2]{park2022viehweg-hyperbolicity},
a morphism $f\colon \sX\rightarrow V$ of varieties, with $V$ smooth, is called \emph{Whitney equisingular} if there exists a Whitney stratification $\sX =\cup_{\alpha\in \Lattice}U_\alpha$ such that $f|_{U_\alpha}\colon U_\alpha\rightarrow V $ is smooth for all $\alpha\in \Lattice$ (see \cite[Definition~1.7]{dimca1992singularities-and-topology} for the definition of a Whitney stratification). In such a family, isolated singularities of the fibers necessarily form an \'etale cover of the base.
Sung Gi Park \cite[Conjecture 1.5]{park2022viehweg-hyperbolicity} proposed that, for Whitney equisingular families $f\colon \sX\rightarrow V$ of KSBA-stable varieties with maximal variation, $V$ should be of log general type.  This is a generalization of a known result for smooth families (\cite{popa2017viehweg-hyperbolicity}),
and we may verify this conjecture for families in $M_c$ using the smooth case.
\begin{thm}\label{thm: hyperbolicity}
Let $0<c\leq 1$.  Let $V$ be smooth quasi-projective variety, and $f\colon \sX\rightarrow V$ a flat Whitney equisingular family of stable surfaces in $M_c$. If $f$ is of maximal variation, then $V$ is of log general type.
\end{thm}
\begin{proof}

We may assume that $\dim V\geq 1$, since the statement is trivially true if $\dim V=0$. Since the family is Whitney equisingular and of maximal variation, the image of moduli map $\phi\colon V\rightarrow M_c$ is contained in one of the positive dimensional strata of the stratification~\ref{eq: stratification}, namely,
$M_c^{\mathrm{K3}}$ or $M_{c}^\II$; see Remark~\ref{rmk: dim strata}.

If the image of $\phi$ lies in $M_c^{\mathrm{K3}}$, since the family is Whitney equisingular and the singularities on the fibers are all klt singularities, we may resolve the singularities simultaneously by blowing up the locus of the singularities, which is \'etale over $V$,
and then blow down the $(-1)$-curves in the fibers, to obtain a family of smooth K3 surfaces over $V$ (cf.~\cite[Theorem~2.10]{kollar1988threefolds-and-deformations}), and $V$ is of log general type by \cite[Theorem~A]{popa2017viehweg-hyperbolicity}.

Now assume that the image of $\phi$ is contained in $M_c^{\II}$. Then $V$ is a curve. Note that the normalization of $M_c^{\II}$ is isomorphic to $\A^1$, and it is not enough to conclude directly by this fact that $V$ is of log general type. Instead, we observe that each surface $X$ in $M_c^{\II}$ is uniquely determined by its  simple elliptic singularity (or the minimal resolution thereof). The simple elliptic singularities on each fiber $X_t$ form a section $\tV$ of $f$, and we may blow up it to obtain a family of smooth elliptic curves over $V$, with maximal variation. It follows again from \cite[Theorem~A]{popa2017viehweg-hyperbolicity} that $V$ is of log general type.
\end{proof}


\newcommand{\etalchar}[1]{$^{#1}$}
\def\cprime{$'$}
\providecommand{\bysame}{\leavevmode\hbox to3em{\hrulefill}\thinspace}
\renewcommand\MR[1]{}
\providecommand{\MRhref}[2]{%
  \href{http://www.ams.org/mathscinet-getitem?mr=#1}{#2}
}

\end{document}